\documentclass[draft,reqno]{amsproc}
\usepackage{amsmath,amssymb}
\usepackage{euscript} 
\usepackage{mathrsfs} 
\usepackage{color,pifont}
\usepackage{graphicx}

\newcommand{\vertiii}[1]{{\left\vert\kern-0.25ex\left\vert\kern-0.25ex\left\vert #1
\right\vert\kern-0.25ex\right\vert\kern-0.25ex\right\vert}}

\newcommand{\isii}[2]{{\left\langle\kern-0.4ex\left\langle
#1, #2 \right\rangle\kern-0.4ex\right\rangle}}

\makeatletter
\@namedef{subjclassname@2010}{%
\textup{2010} Mathematics Subject Classification}
\makeatother

\numberwithin{equation}{subsection}

\newtheorem{thm}{Theorem}[subsection]
\newtheorem{cor}[thm]{Corollary}
\newtheorem{lem}[thm]{Lemma}
\newtheorem{pro}[thm]{Proposition}

\newtheorem*{thm*}{Theorem}
\newtheorem{opq}[thm]{Problem}
\newtheorem*{opq*}{Problem}

\newtheorem*{conj*}{Problem}

\newtheorem{athm}{Theorem}[section]

\newtheorem{alem}[athm]{Lemma}

\theoremstyle{remark}
\newtheorem{rem}[thm]{Remark}

\theoremstyle{definition}
\newtheorem{exa}[thm]{Example}

\newtheorem{dfn}[thm]{Definition}

\DeclareMathOperator{\lin}{\mbox{span}}
\DeclareMathOperator{\E}{e}

\newcommand*{\ascr}{\mathscr{A}}

\newcommand*{\bfrak}{\mathfrak{b}}
\newcommand*{\borel}[1]{{\mathfrak B}(#1)}
\newcommand*{\bscr}{\mathscr{B}}
\newcommand*{\card}[1]{\mathrm{card}(#1)}
\newcommand*{\cbb}{\mathbb C}
\newcommand*{\cfrak}{\mathfrak{c}}
\newcommand*{\D}{\mathrm{d}}
\newcommand*{\dbb}{{\mathbb D}}

\newcommand*{\ee}{\mathcal E}

\newcommand*{\Ge}{\geqslant}
\newcommand*{\gammab}{\boldsymbol{\gamma}}
\newcommand*{\hh}{\mathcal H}

\newcommand*{\I}{{\mathrm i\hspace{.1ex}}}

\newcommand*{\is}[2]{\langle#1,#2\rangle}
\newcommand*{\jd}[1]{\mathscr N(#1)}

\newcommand*{\kk}{\mathcal K}
\newcommand*{\lambdab}{\boldsymbol{\lambda}}
\newcommand*{\Le}{\leqslant}
\newcommand*{\mcal}{\mathcal M}

\newcommand*{\nbb}{\mathbb N}

\newcommand*{\nn}{\mathcal N}

\newcommand*{\omegab}{\boldsymbol{\omega}}
\newcommand*{\ogr}[1]{\boldsymbol B(#1)}
\newcommand*{\ob}[1]{{\mathscr R}(#1)}
\newcommand*{\rbb}{\mathbb R}

\newcommand*{\supp}[1]{\mathrm{supp}(#1)}
\newcommand*{\tbb}{{\mathbb T}}

\newcommand*{\wlam}{W_{\lambdab}}

\newcommand*{\zbb}{\mathbb Z}

\begin{document}
   \title[When is a CPD weighted shift similar to a subnormal
operator?] {When is a CPD weighted shift similar to a
subnormal operator?}
   \author[Z.\ J.\ Jab{\l}o\'nski]{Zenon Jan
Jab{\l}o\'nski}
   \address{Instytut Matematyki,
Uniwersytet Jagiello\'nski, ul.\ \L ojasiewicza 6,
PL-30348 Kra\-k\'ow, Poland}
\email{Zenon.Jablonski@im.uj.edu.pl}
   \author[I.\ B.\ Jung]{Il Bong Jung}
   \address{Department of Mathematics, Kyungpook National University,
Da\-egu 41566, Korea}
   \email{ibjung@knu.ac.kr}
   \author[J.\ Stochel]{Jan Stochel}
\address{Instytut Matematyki, Uniwersytet
Jagiello\'nski, ul.\ \L ojasiewicza 6, PL-30348
Kra\-k\'ow, Poland} \email{Jan.Stochel@im.uj.edu.pl}
   \thanks{The research of the first and third
authors was supported by the National Science Center
(NCN) Grant OPUS No.\ DEC-2021/43/B/ST1/01651. The
research of the second author was supported by Basic
Science Research Program through the National Research
Foundation of Korea (NRF) funded by the Ministry of
Education (NRF-2021R111A1A01043569).}



\maketitle

   \begin{abstract}
We prove that a CPD unilateral weighted
shift $\wlam$ of type III is a quasi-affine
transform of the operator $M_z$ of
multiplication by the independent variable
on the $L^2(\rho)$-closure of analytic
complex polynomials on the complex plane,
where $\rho$ is a measure precisely
determined by $\wlam$. By using this model,
we provide necessary and sufficient
conditions for similarity of $\wlam$ to
$M_z$. Necessary conditions for a CPD
operator to be similar to a subnormal one
are given. A variety of concrete classes of
non-subnormal CPD unilateral weighted shifts
similar to subnormal operators are
established.
   \end{abstract}
   \section{Prerequisites}
   \subsection{Introduction}
The similarity problems for Hilbert space operators
have attracted the attention of many researchers over
many decades and have resulted in many non-trivial
theorems, not to mention one of the earliest, namely
the Sz.-Nagy similarity theorem \cite{Sz-N47}, saying
that an invertible operator is similar to a unitary
operator if and only if both the operator itself and
its inverse are power bounded. Another famous
similarity problem, posed by Halmos
\cite[Problem~6]{hal70} and solved negatively by
Pisier \cite{Pi97}, concerns whether a polynomially
bounded operator is similar to a contraction. Yet
another problem, called the Kadison similarity problem
\cite{Kad55}, asking whether a bounded homomorphism of
a $C^*$-algebra into the algebra of operators on a
Hilbert space is similar to a $*$-homomorphism, has
remained unsolved for almost $70$ years.

The aforementioned similarity problems do indeed apply
to operator-valued functions on algebraic structures.
The two questions described below can be seen more as
similarity problems for single operators. The first,
stated by Halmos \cite[Problem~2]{hal70}, asks whether
every part of a weighted shift is similar to a
weighted shift. The second, in the spirit of our
paper, is whether similarity of two normal operators
implies their unitary equivalence. An affirmative
solution was given by Putnam \cite[Theorem~I]{Put51}.
D. E. Sarason showed that this is no longer true for
subnormal operators, namely he constructed an explicit
example of two similar subnormal unilateral weighted
shifts that are not unitary equivalent (see
\cite[Problem~199]{Hal82}).

Every (injective) unilateral weighted shift
is unitarily equivalent to the operator of
multiplication by the independent variable
on some Hilbert space of formal series (see
\cite[Sec.~3]{shi74}). It is known that
every subnormal unilateral weighted shift is
unitarily equivalent to the operator $M_z$
of multiplication by the independent
variable $z$ on $H^{2}(\mu)$, where $\mu$ is
a finite Borel measure on a compact subset
of the complex plane $\cbb$ and $H^{2}(\mu)$
is the $L^2(\mu)$-closure of analytic
complex polynomials (see
\cite[Theorem~II.6.10]{Con91}). This
together with \cite[Theorem~4.1]{Sto91}
implies that every contractive conditionally
positive definite (CPD) unilateral weighted
shift is unitarily equivalent to $M_z$ on
$H^{2}(\mu)$. In general, a CPD unilateral
weighted shift may not be unitarily
equivalent to $M_z$ on $H^{2}(\mu)$. Worse,
it may not even be similar to such an
operator (see, e.g., Theorem~\ref{feqas} and
Proposition~\ref{nydseq}). A natural
question then arises: when is a CPD
unilateral weighted shift similar to the
operator of multiplication by $z$ on some
$H^{2}(\mu)$? In fact, according to
Lemma~\ref{incsq}, this question is
equivalent to the following problem:
   \begin{opq} \label{s2so}
When is a CPD unilateral weighted shift
similar to a subnormal operator?
   \end{opq}
In this paper, we will analyze
Problem~\ref{s2so}, namely, we will look for
necessary and/or sufficient conditions,
written in terms of the scalar representing
triplet, for a CPD unilateral weighted shift
to be similar to a subnormal operator. It
can be seen that CPD operators are cousins
of subnormal operators in the sense of
harmonic analysis on the additive semigroup
$\zbb_+$ of all nonnegative integers, that
is, a Hilbert space operator $T$ is
subnormal (resp., CPD) if and only if the
sequence $\{\|T^n h\|^2\}_{n=0}^{\infty}$ is
positive definite (resp., conditionally
positive definite) on $\zbb_+$ for every
vector $h$ (see \cite{Con91,Ja-Ju-St22}). To
better understand the similarity problem, we
divide the class of all CPD unilateral
weighted shifts into three disjoint classes
the members of which are called of types I,
II and III, respectively (see
Definition~\ref{kgsw}). CPD unilateral
weighted shifts of class I are exactly
$2$-isometries.

The following theorem is the most general
result obtained by the authors in this area.
On this occasion, we refer the reader to
\eqref{b2t}, \eqref{wtsht} and
Theorem~\ref{Ber-G-W} for the definitions of
$\bscr_2(\cdot)$, $\wlam$ and ``Berger
measure'', respectively. Recall that an
operator from a Hilbert space $\hh$ into a
Hilbert space $\kk$ is called {\em
quasi-invertible} if it is injective and has
dense range.
   \begin{thm} \label{mainth}
Let $\wlam\in \ogr{\ell^2}$ be a CPD
unilateral weighted shift with weights
$\lambdab=\{\lambda_n\}_{n=0}^{\infty}$ that
is not of type~I. Then there exists a
nonzero, compactly supported, finite,
rotation-invariant Borel measure $\rho$ on
$\cbb$ and an operator $X\in \ogr{\ell^2,
H^2(\rho)}$ with dense range such~that
   \begin{enumerate}
   \item[(i)] $X \wlam= M_z X$,
   \item[(ii)] $X^*X=\bscr_2(\wlam)$,
   \item[(iii)] $X$ is
quasi-invertible if and only if $\wlam$ is
of type~III,
   \item[(iv)] if $\wlam$ is of type~III, then $M_z
\in \ogr{H^2(\rho)}$ is unitarily equivalent
to a subnormal unilateral weighted shift
with the Berger measure
$\frac{1}{\rho(\cbb)} \rho\circ \phi^{-1}$,
where $\phi(z)=|z|^2$ for $z\in \cbb$ and
$\rho\circ
\phi^{-1}(\varDelta)=\rho(\phi^{-1}(\varDelta))$
for any Borel subset $\varDelta$ of the
closed half-line $[0,\infty)$,
   \item[(v)] $X$ has a bounded inverse if
and only if there exists $\varepsilon > 0$ such that
$\bscr_2(\wlam) \Ge \varepsilon I$, which is
equivalent to
   \begin{align} \label{bydqa}
\beta_n=\beta_n(\wlam):= 1-2\lambda_n^2 + \lambda_n^2
\lambda_{n+1}^2 \Ge \varepsilon, \quad n\in \zbb_+;
   \end{align}
if this is the case, then $\wlam$ is of type~III and
is similar to a subnormal unilateral weighted shift
with the Berger measure $\frac{1}{\rho(\cbb)}
\rho\circ \phi^{-1}$.
   \end{enumerate}
   \end{thm}
Observe that the condition $\bscr_2(T) \Ge
\varepsilon I$ is not necessary for a CPD
unilateral weighted shift $\wlam$ to be
similar to a subnormal operator, e.g., an
isometry $T$ being subnormal is CPD and
$\bscr_2(T) = 0$. Necessary conditions for a
CPD operator are given in
Subsection~\ref{Sec.2.1} (see
Theorem~\ref{feqas}). As a consequence, it
is proved that a non-isometric $3$-isometry
is not similar to any subnormal operator
(see Corollary~\ref{ns3ic}). The proof of
Theorem~\ref{mainth} is provided in
Subsection~\ref{Sec.2.2}. In this subsection
we also show that non-subnormal CPD weighted
shifts of types I and II are not similar to
subnormal operators (see
Proposition~\ref{nydseq}). In fact, CPD
unilateral weighted shifts of types I and II
satisfy the following dichotomy property:
they are either subnormal or not similar to
a subnormal operator. This dichotomy is no
longer valid for unilateral weighted shifts
of type III (see Theorems~\ref{kdwq} and
\ref{ineqsuf}). As shown in
Proposition~\ref{alevy}, in contrast to
Theorem~\ref{mainth}, a single non-subnormal
CPD unilateral weighted shift is a
quasi-affine transform of a huge number of
subnormal unilateral weighted shifts.
Subsection~\ref{Sec.2.3} provides criteria
for a CPD unilateral weighted shift to be of
type III. In fact, we give a model for CPD
unilateral weighted shifts of types I and II
(see Theorem~\ref{bacyt}). Subnormality of
the model unilateral weighted shift is
characterized in Corollary~\ref{abWn}. In
Subsections~\ref{Sec.3.1} and \ref{Sec.3.2}
we present several concrete classes of
non-subnormal CPD unilateral weighted shifts
that are similar to subnormal operators (see
Theorems~\ref{kdwq} and \ref{ineqsuf}). The
paper concludes with an appendix
characterizing when one unilateral weighted
shift is a quasi-affine transform of
another, and showing that the
quasi-similarity and similarity relations
coincide for unilateral weighted shifts.
   \subsection{A little harmonic analysis}
Let $\rbb$ and $\cbb$ stand for the fields of real and
complex numbers, respectively, and let $\rbb_+=\{x \in
\rbb \colon x \Ge 0\}$, $\dbb=\{z\in \cbb\colon
|z|<1\}$ and $\tbb=\{z\in \cbb\colon |z|=1\}$. Denote
by $\zbb_+$ and $\nbb$ the sets of nonnegative and
positive integers respectively. We write
$\borel{\varOmega}$ for the $\sigma$-algebra of Borel
subsets of a topological space $\varOmega$. Unless
otherwise stated, all scalar measures we consider in
this paper are assumed to be positive. We call a
finite Borel measure $\rho$ on $\cbb$ {\em
rotation-invariant} if $\rho(\E^{\I \varphi}
\varDelta) = \rho(\varDelta)$ for every $\varDelta\in
\borel{\cbb}$ and every $\varphi\in [0,2\pi)$. The
closed support of a finite Borel measure $\mu$ on
$\cbb$ is denoted by $\supp{\mu}$. We write $\delta_t$
for the Borel probability measure on $\rbb$
concentrated at the point $t\in \rbb$. Let $\cbb[z]$
stand for the ring of all polynomials in complex
variable $z$ with complex coefficients. We regard them
as complex functions on $\cbb$.

We now recall some basic facts from harmonic analysis
on semigroups (see \cite{B-C-R}; see also
\cite{Ja-Ju-St22}). Let $\gammab =
\{\gamma_n\}_{n=0}^{\infty}$ be a sequence of real
numbers. We say that $\gammab$ is {\em positive
definite} ({\em PD} for brevity) if
   \begin{align} \label{wicher}
\sum_{i,j=0}^k \gamma_{i+j} \lambda_i \bar\lambda_j
\Ge 0,
   \end{align}
for all finite sequences $\lambda_0, \ldots, \lambda_k
\in \cbb$. If the inequality \eqref{wicher} holds for
all finite sequences $\lambda_0, \ldots, \lambda_k \in
\cbb$ such that $\sum_{j=0}^k \lambda_j=0$, then we
call $\gammab$ {\em conditionally positive definite}
({\em CPD} for brevity). Obviously, PD sequences are
CPD but not conversely. The discrete differentiation
transformation $\triangle\colon \rbb^{\zbb_+} \to
\rbb^{\zbb_+}$ is given by
   \begin{align*}
(\triangle \gammab)_n = \gamma_{n+1} - \gamma_n, \quad
n\in \zbb_+, \, \gammab = \{\gamma_n\}_{n=0}^{\infty}
\in \rbb^{\zbb_+}.
   \end{align*}
We denote by $\triangle^k$ the $k$th composition power
of $\triangle$. Given $n\in \zbb_+$, we define the
polynomial $Q_n$ in real variable $x$ by
   \begin{align} \label{klaud}
Q_n(x) =
   \begin{cases}
0 & \text{if } n=0,1,
   \\
\sum_{j=0}^{n-2} (n -j -1) x^j & \text{if } n\Ge 2,
   \end{cases}
\quad x \in \rbb.
   \end{align}
Then $Q_n$ takes the following explicit form:
   \begin{align} \label{rnx-1}
Q_n(x) & = \frac{x^n-1 - n (x-1)}{(x-1)^2}, \quad n
\in \zbb_+, \, x\in \rbb\setminus \{1\}.
   \end{align}
The polynomials $Q_n$ have the following properties:
   \begin{align} \label{rnx-0}
Q_{n+1}(x) & = x Q_n(x) + n, \quad n \in \zbb_+, \,
x\in \rbb,
   \\ \label{del2}
(\triangle^2 Q_{(\cdot)}(x))_n & = x^n,
\quad n\in \zbb_+, \, x\in \rbb,
   \end{align}
where $\triangle^j Q_{(\cdot)}(x)$ denotes the action
of the transformation $\triangle^j$ on the sequence
$\{Q_{n}(x)\}_{n=0}^{\infty}$ for $x\in \rbb$ and
$j\in \nbb$.

CPD sequences of (at most) exponential growth can be
characterized as follows. The formula \eqref{optwm} is
derived from the L\'{e}vy-Khinchin formula for
negative definite functions on $*$-semigroups (see
\cite[Theorems~4.3.19 and 6.2.6]{B-C-R}).
   \begin{thm}[\mbox{\cite[Theorem~2.2.5]{Ja-Ju-St22}}]
A sequence $\{\gamma_n\}_{n=0}^{\infty}$ of real
numbers is CPD and $\limsup_{n\to
\infty}|\gamma_n|^{1/n} < \infty$ if and only if there
exist $b\in \rbb$, $c\in \rbb_+$ and a finite
compactly supported Borel measure $\nu$ on $\rbb$ such
that $\nu(\{1\})=0$ and
   \begin{align} \label{optwm}
\gamma_n = \gamma_0 + bn + c n^2 + \int_{\rbb} Q_n(x)
\D\nu(x), \quad n\in \zbb_+.
   \end{align}
Such a triplet $(b,c,\nu)$ is unique.
   \end{thm}
   \subsection{A little operator theory}
Given (complex) Hilbert spaces $\hh$ and
$\kk$, we denote by $\ogr{\hh,\kk}$ the
Banach space of all bounded linear operators
from $\hh$ to $\kk$. We abbreviate
$\ogr{\hh,\hh}$ to $\ogr{\hh}$ and regard
$\ogr{\hh}$ as a unital $C^*$-algebra with
the identity operator $I$ on $\hh$ as the
unit. The kernel and the range of $T\in
\ogr{\hh,\kk}$ are denoted by $\jd{T}$ and
$\ob{T}$, respectively. Two operators $A\in
\ogr{\hh}$ and $B\in \ogr{\kk}$ are said to
be {\em similar} if there exists a bijective
operator $X \in \ogr{\hh,\kk}$ such that
$XA=B X$. By the inverse mapping theorem,
$X^{-1}\in \ogr{\kk,\hh}$. We say that an
operator $X \in \ogr{\hh,\kk}$ is {\em
quasi-invertible}, or that $X$ is a {\em
quasi-affinity}, if $X$ is injective and has
dense range. Following \cite{Sz-F-B-K10} we
say that the operator $A$ is a {\em
quasi-affine transform} of the operator $B$
if there exists a quasi-invertible operator
$X\in \ogr{\hh,\kk}$ such that $XA=BX$. The
operators $A$ and $B$ are called {\em
quasi-similar} if they are quasi-affine
transforms of each another. Obviously,
similarity implies quasi-similarity (see
\cite{Sz-F-B-K10} for more information).

An operator $T\in \ogr{\hh}$ is said to be
{\em power bounded} if $\sup_{n\in \zbb_+}
\|T^n\|<\infty$. By the uniform boundedness
principle, an operator $T\in \ogr{\hh}$ is
power bounded if and only if $\limsup_{n\to
\infty} \|T^n h\|<\infty$ for every $h\in
\hh$. Given an operator $T\in \ogr{\hh}$ and
an integer $m\Ge 1$, we set
   \begin{align} \label{b2t}
\bscr_m(T) = \sum_{k=0}^m (-1)^k \binom{m}{k}
T^{*k}T^k.
   \end{align}
We say that $T$ is am $m$-isometry if $\bscr_m(T)=0$.
We refer the reader to \cite{Ag-St1,Ag-St2,Ag-St3} for
the basics of the theory of $m$-isometries.

   An operator $T\in \ogr{\hh}$ is said to be {\em
subnormal} if there exists a Hilbert space $\kk$ and a
normal operator $N\in \ogr{\kk}$ such that
$\hh\subseteq \kk$ (isometric embedding) and $Th=Nh$
for every $h\in \hh$. If $\kk$ has no proper closed
subspace that reduces $N$ and contains $\hh$, then $N$
is called a {\em minimal normal extension} of $T$.
According to Lambert's theorem $T$ is subnormal if and
only if for all $h\in \hh$, the sequence $\{\|T^n
h\|^2\}_{n=0}^{\infty}$ is positive definite (see
\cite{lam} and \cite[Theorem~1.1.1]{Ja-Ju-St22}).
Following \cite{Ja-Ju-St22}, we say that an operator
$T\in \ogr{\hh}$ is {\em conditionally positive
definite} ({\em {CPD}} for brevity) if for all $h\in
\hh$, the sequence $\{\|T^n h\|^2\}_{n=0}^{\infty}$ is
conditionally positive definite. The class of CPD
operators contains in particular subnormal operators
\cite{Hal50,Con91}, complete hypercontractions of
order $2$ \cite{Cha-Sh}, $3$-isometries
\cite{Ag-St1,Ag-St2,Ag-St3} and many~others.

Let $F\colon \ascr \to \ogr{\hh}$ be a {\em
semispectral measure} on a $\sigma$-algebra $\ascr$ of
subsets of a set $\varOmega$, i.e., $F$ is
$\sigma$-additive in the weak operator topology and
$F(\varDelta)\Ge 0$ for every $\varDelta\in \ascr$. We
{\em do not assume} that $F(\varOmega)=I$. Denote by
$L^1(F)$ the linear space of all complex
$\ascr$-measurable functions $\zeta$ on $\varOmega$
such that $\int_{\varOmega} |\zeta(x)| \is{F(\D x)h}h
< \infty$ for all $h\in \hh$. Then for every $\zeta\in
L^1(F)$, there exists a unique operator
$\int_\varOmega \zeta \D F \in \ogr{\hh}$ such that
(see e.g., \cite[Appendix]{Sto92})
   \begin{align*}
\Big\langle\int_\varOmega \zeta \D F h, h\Big\rangle =
\int_\varOmega \zeta(x) \is{F(\D x)h}h, \quad h\in\hh.
   \end{align*}
If $\varOmega=\rbb, \cbb$ and $F\colon
\borel{\varOmega} \to \ogr{\hh}$ is a semispectral
measure, then its closed support is denoted by
$\supp{F}$ (recall that such $F$ is automatically
regular so $\supp{F}$ exists). By a {\em semispectral
measure} of a subnormal operator $T\in \ogr{\hh}$ we
mean a compactly supported semispectral measure
$G\colon \borel{\cbb} \to \ogr{\hh}$ defined by
$G(\varDelta) = PE(\varDelta)|_{\hh}$ for $\varDelta
\in\borel{\cbb}$, where $E\colon \borel{\cbb} \to
\ogr{\kk}$ is the spectral measure of a minimal normal
extension $N\in \ogr{\kk}$ of $T$ and $P\in \ogr{\kk}$
is the orthogonal projection of $\kk$ onto $\hh$.
Clearly $F$ is normalized, i.e., $F(\cbb)=I$. It
follows from \cite[Proposition~5]{Ju-St08} and
\cite[Proposition~II.2.5]{Con91} that a subnormal
operator has exactly one semispectral measure. It is
also easily seen that (see, e.g.,
\cite[Proposition~3]{Ju-St08})
   \begin{align}  \label{tobemom}
T^{*m}T^n = \int_{\cbb} \bar z^{m}z^{n} G(\D z), \quad
m,n\in \zbb_+.
   \end{align}
We refer the reader to \cite{Con91} for the
foundations of the theory of subnormal operators.

The CPD operators can be characterized as follows.
   \begin{thm}[{\cite[Theorem~3.1.1]{Ja-Ju-St22}}]
\label{cpdops} An operator $T\in \ogr{\hh}$ is CPD if
and only if there exist operators $B,C\in \ogr{\hh}$
and a compactly supported semispectral measure
$F\colon \borel{\rbb_+} \to \ogr{\hh}$ such that
$B=B^*$, $C\Ge 0$, $F(\{1\})=0$ and
   \begin{align} \label{tynst}
T^{*n}T^n = I + n B + n^2 C + \int_{\rbb_+} Q_n(x)
F(\D x), \quad n\in \zbb_+.
   \end{align}
Moreover, such a triplet $(B,C,F)$ is unique.
   \end{thm}
Call $(B,C,F)$ appearing in Theorem~\ref{cpdops} the
{\em representing triplet} of $T$.
   \subsection{CPD weighted shifts}
Let $\ell^2$ stand for the Hilbert space of square
summable complex sequences
$\{\alpha_n\}_{n=0}^{\infty}$ and let
$\{e_n\}_{n=0}^{\infty}$ be the standard orthonormal
basis of $\ell^2$, i.e.,
$e_n=\{\delta_{n,k}\}_{k=0}^{\infty}$, where
$\delta_{n,k}$ is the Kronecker delta. For any bounded
sequence $\lambdab=\{\lambda_n\}_{n=0}^{\infty}
\subseteq (0,\infty)$, there exits a unique operator
$\wlam\in \ogr{\ell^2}$ such that
   \begin{align} \label{wtsht}
\wlam e_n = \lambda_n e_{n+1}, \quad n \in \zbb_+.
   \end{align}
We call $\wlam$ a {\em unilateral weighted
shift} with weights $\lambdab$. Clearly,
such $\wlam$ is always injective. If
$\lambda_n=1$ for all $n \in \zbb_+$, then
we call $\wlam$ the {\em unilateral shift}.
In this paper, we only consider bounded
unilateral weighted shifts with positive
real weights. We refer the reader to the
classical treatise \cite{shi74} for the
fundamentals of the theory o weighted
shifts.

If $\wlam\in \ogr{\ell^2}$ is a unilateral weighted
shift, then the sequence
$\hat{\lambdab}=\{\hat\lambda_n\}_{n=0}^{\infty}$
associated to $\wlam$ is defined by
   \begin{align} \label{mur-hupy}
\hat{\lambda}_n =
   \begin{cases}
1 & \text{if } n=0,
   \\
\lambda_0^2 \cdots \lambda_{n-1}^2 & \text{if } n \Ge
1,
   \end{cases}
\quad n\in \zbb_+.
   \end{align}
Call the sequence $\hat{\lambdab}$ the {\em formal
moment sequence} of $\wlam$. The famous Berger
theorem, which characterizes the subnormality of
unilateral weight\-ed shifts, can be stated as
follows.
   \begin{thm}[\mbox{\cite{g-w70,hal70}};
\mbox{\cite[Proposition~25]{shi74}}]
\label{Ber-G-W} $\wlam$ is subnormal if and
only if there exists a $($unique$)$
compactly supported Borel probability
measure $\mu$ on $\rbb_+$ such that
   \begin{align}  \label{Stieq}
\hat \lambda_n = \int_{\rbb_+} x^n \D \mu(x), \quad n
\in \zbb_+.
   \end{align}
Moreover $\supp{\mu}\subseteq
[0,\|\wlam\|^2]$. The measure $\mu$ is
called the Berger measure of~$\wlam$.
   \end{thm}
CPD unilateral weighted shifts can be characterized as
follows.
   \begin{thm}[\mbox{\cite[Theorem~3.1]{Ja-Ju-Le-St23}}]
\label{cpdws} A unilateral weighted shift
$\wlam$ is CPD if and only if there exist
$\bfrak\in \rbb$, $\cfrak\in \rbb_+$ and a
compactly supported finite Borel measure
$\nu$ on $\rbb_+$ such that $\nu(\{1\})=0$
and
   \begin{align} \label{wnezero}
\hat\lambda_n = 1 + \bfrak n + \cfrak n^2 +
\int_{\rbb_+} Q_n \D\nu, \quad n\in \zbb_+.
   \end{align}
Moreover, the triplet $(\bfrak,\cfrak,\nu)$ is unique.
   \end{thm}
Call $(\bfrak,\cfrak,\nu)$ appearing in
Theorem~\ref{cpdws} the {\em scalar representing
triplet} of $\wlam$. The relationship between
representing triplets $(B,C,F)$ and scalar
representing triplets $(\bfrak,\cfrak,\nu)$ of CPD
unilateral weighted shifts is given below.
   \begin{thm}[\mbox{\cite[Theorem~6.1]{Ja-Ju-Le-St23}}]
\label{truplyt} Let $(B,C,F)$ be the representing
triplet of $\wlam$ and $(\bfrak,\cfrak,\nu)$ be the
scalar representing triplet of $\wlam$. Then $B$, $C$
and $F(\varDelta)$, where $\varDelta \in
\borel{\rbb_+}$, are diagonal operators with respect
to $\{e_k\}_{k=0}^{\infty}$ with diagonal terms
$\{\bfrak_k\}_{k=0}^{\infty}$,
$\{\cfrak_k\}_{k=0}^{\infty}$ and
$\{\nu_k(\varDelta)\}_{k=0}^{\infty}$ given by
   \begin{align*}
\bfrak_k=\frac{\hat\lambda_{k+1} - \hat\lambda_{k}
-\cfrak}{\hat\lambda_k}, \quad \cfrak_k =
\frac{\cfrak}{\hat\lambda_k}, \quad \nu_k(\varDelta) =
\frac{1}{\hat\lambda_k} \int_{\varDelta} x^k \D
\nu(x),
   \end{align*}
where $\hat\lambda_k$ are as in \eqref{mur-hupy}. In
particular, $\cfrak=\is{Ce_0}{e_0}$ and
$\nu(\cdot)=\is{F(\cdot)e_0}{e_0}$.
   \end{thm}
   \begin{cor} \label{bcny}
Under the hypotheses of Theorem~{\em \ref{truplyt}},
$\cfrak=0$ if and only if $C=0$ and
$\supp{\nu}=\supp{F}$.
   \end{cor}
As shown below, CPD unilateral weighted shifts can be
modelled by CPD sequences with positive terms.
   \begin{thm}[\mbox{\cite[Theorem~4.1]{Ja-Ju-Le-St23}}]
\label{wkwcpdws} Let
$\gammab=\{\gamma_{n}\}_{n=0}^{\infty} \subseteq
(0,\infty)$ be of the form
   \begin{gather}
\label{pfsw} \gamma_{n} = 1 + b n + c n^2 +
\int_{\rbb_+} Q_n \D\nu, \quad n\in \zbb_+,
   \end{gather}
where $b\in \rbb$, $c\in \rbb_+$ and $\nu$ is a
compactly supported finite Borel measure on $\rbb_+$
such that $\nu(\{1\})=0$. Then the sequence
$\lambdab=\{\lambda_n\}_{n=0}^{\infty}$ defined by
   \begin{align*}
\lambda_n=\sqrt{\frac{\gamma_{n+1}}{\gamma_{n}}},
\quad n\in \zbb_+,
   \end{align*}
is bounded and the unilateral weighted shift $\wlam$
is CPD. Moreover, $\hat\lambdab=\gammab$.
   \end{thm}
The detailed discussion of the circumstances under
which a given sequence $\{\gamma_{n}\}_{n=0}^{\infty}$
of the form \eqref{pfsw} is positive is performed in
\cite[Theorem~5.2]{Ja-Ju-Le-St23}. To avoid
over-complexity of the problem studied in this paper,
we will focus on the case when $b\Ge 0$ (then
automatically $\gamma_n \Ge 1$ for all $n\in \zbb_+$).

The next result gives a necessary and sufficient
condition for a CPD unilateral weighted shift to be
subnormal.
   \begin{pro} \label{rewa}
Let $\wlam\in \ogr{\ell^2}$. Then the following are
equivalent{\em :}
   \begin{enumerate}
   \item[\rm(i)] $\wlam$ is subnormal,
   \item[\rm(ii)] $\wlam$ is CPD with the scalar
representing triplet $(\bfrak,\cfrak,\nu)$ such
that\footnote{If (ii-a) holds, then by the
Cauchy-Schwarz inequality \mbox{$\frac{1}{x-1}\in
L^1(\nu)$}, so $\int_{\rbb_+} \frac{1}{x-1} \D \nu(x)$
exists.}{\em :}
   \begin{enumerate}
   \item[\rm(ii-a)] $\int_{\rbb_+} \frac{1}{(x-1)^2} \D \nu(x) \Le
1$,
   \item[\rm(ii-b)] $\bfrak=\int_{\rbb_+} \frac{1}{x-1} \D
\nu(x)$,
   \item[\rm(ii-c)] $\cfrak=0$.
   \end{enumerate}
   \end{enumerate}
   \end{pro}
   \begin{proof}
(i)$\Rightarrow$(ii) Assume that $\wlam$ is subnormal.
Then, by Theorem~\ref{Ber-G-W}, $\hat{\lambdab}$ is
PD, so it is CPD. Applying
\cite[Theorem~2.2.12]{Ja-Ju-St22} to
$\gammab:=\hat{\lambdab}$, we get (ii).

(ii)$\Rightarrow$(i) Suppose (ii) is valid. Applying
\cite[Theorem~2.2.12]{Ja-Ju-St22} to
$\gammab:=\hat{\lambdab}$, we see that
$\hat{\lambdab}$ is PD. According to the Hamburger
theorem (see \cite[Theorem~6.2.2]{B-C-R}), there
exists a finite Borel measure $\mu$ on $\rbb$ such
that
   \begin{align} \label{gmyn}
\hat\lambda_n=\int_{\rbb} x^n \D \mu(x), \quad n\in
\zbb_+.
   \end{align}
By \eqref{del2} and \eqref{wnezero}, we have (cf.\
\cite[(2.2.16)]{Ja-Ju-St22})
   \begin{align} \label{tryjg}
(\triangle^2 \hat{\lambdab})_n = \int_{\rbb_+} x^n \D
(\nu + 2\cfrak \delta_1)(x), \quad n\in \zbb_+.
   \end{align}
Using \eqref{gmyn} and \eqref{tryjg}, we deduce that
   \begin{align*}
\int_{\rbb} x^n (x-1)^2 \D \mu(x)=(\triangle^2
\hat\lambdab)_n = \int_{\rbb_+} x^n \D (\nu + 2\cfrak
\delta_1)(x), \quad n\in \zbb_+.
   \end{align*}
Since the measure $\nu + 2\cfrak \delta_1$ is
compactly supported, the sequence $\{(\triangle^2
\hat\lambdab)_n\}_{n=0}^{\infty}$ is a determinate
Hamburger moment sequence (see
\cite[Corollary~4.2]{Sch17}). Hence
   \begin{align*}
\int_{\varDelta}(x-1)^2 \D \mu(x)=\nu(\varDelta) +
2\cfrak \delta_1(\varDelta), \quad \varDelta\in
\borel{\rbb}.
   \end{align*}
This implies that $\supp{\mu}$ is a compact subset of
$\rbb_+$. Therefore, by \eqref{gmyn}, we have
   \begin{align*}
\hat\lambda_n=\int_{\rbb_+} x^n \D \mu(x), \quad n\in
\zbb_+.
   \end{align*}
Using Theorem~\ref{Ber-G-W}, we conclude that $\wlam$
is subnormal.
   \end{proof}
Let $\rho$ be a compactly supported finite
Borel measure on $\cbb$. Denote by
$H^2(\rho)$ the closure in $L^2(\rho)$ of
the set of all polynomials in indeterminate
$z$ with complex coefficients. The operator
of {\em multiplication} by the independent
variable ``$z$'', denoted by $M_z$, is
defined on $H^2(\rho)$ by
   \begin{align*}
(M_z f)(w) = w f(w) \;\; \text{for $\rho$-a.e.\ $w$},
\quad f \in H^2(\rho).
   \end{align*}
Clearly, $M_z$ is a bounded operator on $H^2(\rho)$,
which is cyclic with the identity function on $\cbb$
as the cyclic vector. Moreover, since the operator of
multiplication by the independent variable on
$L^2(\rho)$ is normal and extends $M_z$, we see that
$M_z$ is subnormal.
   \section{Similarity}
   \subsection{\label{Sec.2.1}Necessary conditions}
In this subsection, we provide some
necessary conditions for a CPD operator to
be similar to a subnormal operator. We
formulate them in terms of the representing
triplet.
   \begin{lem} \label{jswq}
Let $T\in \ogr{\hh}$ be a CPD operator with the
representing triplet $(B,C,F)$ such that $\supp{F}
\subseteq [0,1]$. Suppose that $T$ is is similar to a
subnormal operator. Then $T$ is power bounded.
   \end{lem}
   \begin{proof}
Suppose that there exists a subnormal operator $S\in
\ogr{\kk}$ and an invertible operator $X\in
\ogr{\kk,\hh}$ such that $XS=TX$. Then
   \begin{align*}
XS^n=T^nX, \quad n\in \zbb_+.
   \end{align*}
Using \eqref{tynst} and the fact that $X$ is
invertible, we get
   \allowdisplaybreaks
   \begin{align}  \notag
\|X\|^{-2} \varLambda(n;h) & = \|X\|^{-2} \|T^nXh\|^2
=\|X\|^{-2}\|XS^nh\|^2
   \\ \notag
& \Le \|S^n h\|^2
   \\ \label{xsmy}
& \Le \|X^{-1}\|^2 \|XS^nh\|^2 = \|X^{-1}\|^2
\varLambda(n;h), \quad n\in \zbb_+, \, h\in \hh,
   \end{align}
where $\varLambda(n;h)$ is defined for $h\in \hh$ and
$n\in \zbb_+$ as follows
   \begin{align*}
\varLambda(n;h):=\|Xh\|^2 + n \is{BXh}{Xh} + n^2
\is{CXh}{Xh} + \int_{\rbb_+} Q_n(x) \is{F(\D
x)Xh}{Xh}.
   \end{align*}
Let $G$ be the semispectral measure of the subnormal
operator $S$. Then, by \eqref{tobemom} and
\eqref{xsmy}, we have
   \begin{align} \label{oszweq}
\|X\|^{-2} \varLambda(n;h) \Le \int_{\cbb} |z|^{2n}
\is{G(\D z)h}{h} \Le \|X^{-1}\|^2 \varLambda(n;h),
\quad n\in \zbb_+, \, h\in \hh.
   \end{align}
Since, by \eqref{klaud}, $Q_n(x) \Le \frac{(n-1)n}{2}$
for all $x\in [0,1]$ and $n\in \zbb_+$, we deduce that
   \begin{align*}
\varLambda(n;h) \Le \|Xh\|^2 + n \is{BXh}{Xh} + n^2
\is{CXh}{Xh} + \frac{(n-1)n}{2} \is{F(\rbb_+)Xh}{Xh}
   \end{align*}
for all $h\in \hh$ and $n\in \zbb_+$. This and the
right inequality in \eqref{oszweq} imply that
   \begin{align*}
\lim_{n\to\infty}\Big(\int_{\cbb} |z|^{2n} \is{G(\D
z)h}{h}\Big)^{1/n} \Le 1, \quad h\in \hh,
   \end{align*}
so, by \cite[Ex.\ 4(e), p.\ 71]{Rud87}, we have
   \begin{align} \label{sykppi}
\supp{\is{G(\cdot)h}{h}} \subseteq \bar\dbb, \quad
h\in \hh.
   \end{align}
Using the left inequality in \eqref{oszweq}, together
with \eqref{sykppi} and Lebesgue's dominated
convergence theorem, we obtain
   \begin{align*}
\|X\|^{-2} \limsup_{n\to\infty} \varLambda(n;h) \Le
\lim_{n\to \infty} \int_{\cbb}|z|^{2n} \is{G(\D
z)h}{h} = \is{G(\tbb)h}{h}, \quad h\in \hh.
   \end{align*}
Hence $\limsup_{n\to\infty} \|T^n h\|^2
\overset{\eqref{tynst}} =
\limsup_{n\to\infty}\varLambda(n;h) < \infty$ for
every $h\in \hh$. By the uniform boundedness
principle, $T$ is power bounded.
   \end{proof}
   \begin{rem}
As for the last part of the proof of Lemma~\ref{jswq},
starting from \eqref{sykppi}, we can also argue as
follows. Applying \eqref{tobemom} to $n=1$ and using
\eqref{sykppi}, we conclude that $S$ is a contraction
and therefore power bounded. Hence, since $T$ is
similar to $S$, $T$ is also power bounded.
   \hfill $\diamondsuit$
   \end{rem}
   \begin{thm} \label{feqas}
Let $T\in \ogr{\hh}$ be a CPD operator with the
representing triplet $(B,C,F)$ such that $\supp{F}
\subseteq [0,1]$. Suppose that $T$ is is similar to a
subnormal operator. Then the following conditions
hold{\em :}
   \begin{enumerate}
   \item[(i)] $C = 0$,
   \item[(ii)] $B+F(\rbb_+)\Le 0$,
   \item[(iii)] $B+F(\rbb_+) = 0$ or
$\supp{F} \not \subseteq \{0\}$,
   \item[(iv)] $B\not\Ge
0$ or $\supp{F} \cap (0,1) = \emptyset$.
   \end{enumerate}
   \end{thm}
   \begin{proof}
   (i) Suppose, to the contrary, that $C\neq 0$. Since
$C \Ge 0$, there exists $h_0\in \hh$ such that
$\is{Ch_0}{h_0} > 0$. Noting that
   \begin{multline*}
\|T^n h_0\|^2 \overset{\eqref{tynst}} = \|h_0\|^2 + n
\is{Bh_0}{h_0} + n^2 \is{Ch_0}{h_0} + \int_{\rbb_+}
Q_n(x) \is{F(\D x)h_0}{h_0}
   \\
\Ge n \is{Bh_0}{h_0} + n^2 \is{Ch_0}{h_0}, \quad n\in
\zbb_+,
   \end{multline*}
we see that $\lim_{n\to\infty} \|T^n h_0\| =+\infty$,
which contradicts Lemma~\ref{jswq}.

(ii) Suppose, to the contrary, that
$B+F(\rbb_+)\nleqslant 0$. Then there exists $h_0 \in
\hh$ such that
   \begin{align*}
\alpha:=\is{(B+F(\rbb_+)h_0}{h_0} > 0.
   \end{align*}
By (i), $C=0$. Since, by \eqref{klaud}, $Q_n(x) \Ge
n-1$ for all $x\in \rbb_+$ and $n\Ge 1$, we see that
   \begin{align*}
\|T^n h_0\|^2 & \Ge \|h_0\|^2 + n \is{Bh_0}{h_0} +
(n-1) \is{F(\rbb_+)h_0}{h_0}
   \\
& = (\|h_0\|^2 - \is{F(\rbb_+)h_0}{h_0}) + n \alpha,
\quad n \Ge 1.
   \end{align*}
so $\lim_{n\to\infty} \|T^n h_0\|=+\infty$, which
contradicts Lemma~\ref{jswq}.

(iii) Suppose, to the contrary, that $B+F(\rbb_+) \neq
0$ and $\supp{F} \subseteq \{0\}$. By (i), $C=0$.
According to \eqref{klaud}, $Q_n(0)=n-1$ for every
$n\Ge 1$, so
   \begin{align} \notag
\|T^n h\|^2 & = \|h\|^2 + n \is{Bh}{h} + (n-1)
\is{F(\{0\})h}{h}
   \\  \label{jtwjda}
& = \is{(I - F(\rbb_+)) h}{h}) + n
\is{(B+F(\rbb_+)h}{h}, \quad h \in \hh, \, n
\Ge 1.
   \end{align}
This implies that $B+F(\rbb_+) \Ge 0$. Since
$B+F(\rbb_+) \neq 0$, it must be $B+F(\rbb_+)\not\Le
0$, which contradicts (ii).

(iv) Suppose, to the contrary, that $B\Ge 0$
and there exists $\theta \in \supp{F} \cap
(0,1)$. By (i), $C=0$. Let $\varepsilon>0$
be such that $J_{\varepsilon}:=(\theta -
\varepsilon,\theta +\varepsilon) \subseteq
(0,1)$. Then $F(J_{\varepsilon}) \neq 0$,
which implies that there exists $h_0\in \hh$
such that $\is{F(J_{\varepsilon})h_0}{h_0} >
0$. This yields
   \begin{align*}
\|T^n h_0\|^2 & \Ge \|h_0\|^2 + n \is{Bh_0}{h_0} +
\int_{J_{\varepsilon}} Q_n(x) \is{F(\D x)h_0}{h_0}
   \\
& \hspace{-1ex}\overset{\eqref{rnx-1}} \Ge
\frac{n (1-(\theta+\varepsilon))
-1}{(1-(\theta-\varepsilon))^2} \,
\is{F(J_{\varepsilon})h_0}{h_0}, \quad n \in
\zbb_+.
   \end{align*}
Again we obtain $\lim_{n\to\infty} \|T^n
h_0\|=+\infty$, which contradicts Lemma~\ref{jswq}.
   \end{proof}
   \begin{rem}
As for condition (i) of Theorem~\ref{feqas},
note that if $\supp{F}$ is not contained in
the closed interval $[0,1]$, it can happen
that the unilateral weighted shift $\wlam$
is similar to a subnormal operator and $C
\neq 0$. Such a possibility is guaranteed by
the ''moreover'' parts of
Theorems~\ref{kdwq} and \ref{ineqsuf} (see
also Corollary~ \ref{bcny}).
   \end{rem}
   \begin{cor} \label{symfo}
Let $T\in \ogr{\hh}$ be a non-subnormal CPD
operator with the representing triplet
$(B,C,F)$ such that $\supp{F} \subseteq
\{0\}$. Then $T$ is not similar to any
subnormal operator.
   \end{cor}
   \begin{proof}
Suppose, to the contrary, that $T$ is
similar to a subnormal operator. By
Theorem~\ref{feqas}, $C=0$ and $B+F(\rbb_+)
= 0$. Hence, by \eqref{jtwjda}, $F(\rbb_+)
\Le I$. This implies that $T$ satisfies
statement (ii) of
\cite[Theorem~3.4.1]{Ja-Ju-St22}. By this
theorem, $T$ is subnormal, a contradiction.
   \end{proof}
   \begin{cor}  \label{ns3ic}
A non-isometric $3$-isometry is not similar
to any subnormal operator.
   \end{cor}
   \begin{proof}
Let $T\in \ogr{\hh}$ be a non-isometric
$3$-isometry. In view of
\cite[Proposition~4.5]{Sh-At00}, $T$ is not
subnormal. By
\cite[Proposition~4.3.1]{Ja-Ju-St22}, $T$ is
CPD and $F=0$. According to
Corollary~\ref{symfo}, $T$ is not similar to
any subnormal operator.
   \end{proof}
   \subsection{\label{Sec.2.2}The proof of the main theorem}
The following lemma is a direct consequence of
\cite[Proposition~4.3.1]{Ja-Ju-St22} and the fact that
$2$-isometries are $3$-isometries (see \cite[p.\
389]{Ag-St1}).
   \begin{lem} \label{dqaer}
Let $T\in \ogr{\hh}$. Then
   \begin{enumerate}
   \item[(i)] $T$ is a
$3$-isometry if and only if $T$ is CPD and $F=0$,
   \item[(ii)] $T$ is a
$2$-isometry if and only if $T$ is CPD, $F=0$ and
$C=0$,
   \end{enumerate}
where $(B,C,F)$ stands for the representing triplet of
a CPD operator $T$.
   \end{lem}
Let us divide the class of all CPD unilateral weighted
shifts into three disjunctive types. That the
definition of types is correct, and the types
themselves are logically disjunctive, can be readily
deduced from Corollary~\ref{bcny} and
Lemma~\ref{dqaer}.
   \begin{dfn} \label{kgsw}
A CPD unilateral weighted shift $\wlam\in
\ogr{\ell^2}$ with the scalar representing triplet
$(\bfrak,\cfrak,\nu)$ is
   \begin{enumerate}
   \item[$\bullet$] of type~I if $\nu=0$ and $\cfrak =
0$, or equivalently if and only if $\wlam$ is a
$2$-isometry,
   \item[$\bullet$] of type~II if $\supp{\nu} =
\{0\}$ and $\cfrak = 0$,
   \item[$\bullet$] of type~III if $\supp{\nu}\nsubseteq \{0\}$ or
$\cfrak \neq 0$.
   \end{enumerate}
   \end{dfn}
Unilateral weighted shifts of type~II are particular
instances of general CPD operators for which
$\supp{F}=\{0\}$ and $C=0$. The latter class of
operators was studied in \cite[Sec.~4.3]{Ja-Ju-St22}.
Its algebraic characterizations were given in
\cite[Proposition~4.3.5]{Ja-Ju-St22}.
   \begin{pro} \label{nydseq}
Non-subnormal CPD weighted shifts of types I
and II are not similar to subnormal
operators.
   \end{pro}
   \begin{proof}
It follows from Definition~\ref{kgsw} that $\supp{\nu}
\subseteq \{0\}$. In view of Corollary~\ref{bcny},
$\supp{F} \subseteq \{0\}$, so Corollary~\ref{symfo}
completes the proof.
   \end{proof}
We refer the reader to Theorem~\ref{bacyt}, which
implicitly contains the model $W_{a,b}$ for CPD
unilateral weighted shifts of both type I and type II.
Such $W_{a,b}$ is of type I (resp., type II) if and
only $\theta=0$ (resp., $\theta > 0$), where
$\theta:=1 - 2 a + a b$. Moreover, by
Corollary~\ref{abWn}, if $W_{a,b}$ is of type I, then
$W_{a,b}$ is subnormal if and only if $a=b=1$
(equivalently: $W_{a,b}$ is an isometry). In turn, if
$W_{a,b}$ is of type II, then $W_{a,b}$ is subnormal
if and only if $a<b=1$. Thus, in view of
Proposition~\ref{nydseq}, we arrived at the following
dichotomy property.
   \begin{align*}
   \begin{minipage}{75ex}
{\em If $\wlam\in \ogr{\ell^2}$ is a CPD unilateral
weighted shift of type I or type II, then it is either
subnormal or it is not similar to a subnormal
operator.}
   \end{minipage}
   \end{align*}
This dichotomy is no longer valid for unilateral
weighted shifts of type III (see Theorems~\ref{kdwq}
and \ref{ineqsuf}, and Example~\ref{3uwre}).

   Now we are ready to prove the main result of this
paper.
   \begin{proof}[Proof of Theorem~\ref{mainth}]
Consider first the case of an arbitrary CPD operator
$T\in \ogr{\hh}$. Let $(B,C,F)$ be the representing
triplet of $T$. Then, we have
   \begin{align*}
\|T^n h\|^2 \overset{\eqref{tynst}}= \|h\|^2 + n
\is{Bh}{h} + n^2 \is{Ch}{h} + \int_{\rbb_+} Q_n(x)
\is{F(\D x)h}{h}, \quad n\in \zbb_+.
   \end{align*}
Applying $\triangle^2$ to both sides of the above
equality and using \eqref{del2}, we obtain
   \begin{multline*}
\is{T^{*n}T^n h}{h} - 2 \is{T^{*(n+1)}T^{n+1} h}{h} +
\is{T^{*(n+2)}T^{n+2} h}{h}
   \\
= 2 \is{C h}{h} + \int_{\rbb_+} x^n \is{F(\D x)h}{h},
\quad h\in \hh, \, n \in \zbb_+.
   \end{multline*}
This can be rewritten as follows:
   \begin{align} \label{gynfrw}
\is{\bscr_2(T) T^n h}{T^n h} = 2 \is{C h}{h} +
\int_{\rbb_+} x^n \is{F(\D x)h}{h}, \quad h\in \hh, \,
n \in \zbb_+.
   \end{align}

Suppose now that $T=\wlam$. At this point it is worth
noting that the proof of conditions (i) and (ii) can
be carried out without the assumption that $\wlam$ is
not of type~I. Then, however, the measures $\mu_0$,
$\mu_1$ and $\rho$ defined below can be zero.

(i) Let $(\bfrak, \cfrak, \nu)$ be the scalar
representing triplet of $\wlam$. It follows from
Theorem~\ref{truplyt} and \eqref{gynfrw} that
   \begin{align} \label{yzldwa}
\is{\bscr_2(\wlam) \wlam^n e_0}{\wlam^n e_0} =
\int_{\rbb_+} x^n \D \mu_0(x), \quad h\in \hh, \, n
\in \zbb_+,
   \end{align}
where $\mu_0$ is the finite Borel measure on $\rbb_+$
defined by
   \begin{align} \label{mam0}
\mu_0(\varDelta) = \nu(\varDelta) + 2 \cfrak \delta_1
(\varDelta), \quad \varDelta \in \borel{\rbb_+}.
   \end{align}
Note that $\mu_0\neq 0$. Indeed, otherwise by
\eqref{yzldwa} and the fact that $\bscr_2(\wlam) \Ge
0$ (see \cite[Corollary~3.2.7(i)]{Ja-Ju-St22}),
$\bscr_2(\wlam) \wlam^n e_0=0$ for every $n\in
\zbb_+$, so $\bscr_2(\wlam)=0$, which contradicts the
assumption that $\wlam$ is not of type~I.

Regarding the operator $\bscr_2(\wlam)$, it is routine
to show that
   \begin{align} \label{zewq}
\bscr_2(\wlam) e_n = \beta_n e_n, \quad n \in \zbb_+.
   \end{align}
Since $\bscr_2(\wlam)\Ge 0$, we see that $\beta_n \Ge
0$ for every $n\in \zbb_+$. By \eqref{mur-hupy} and
\eqref{zewq}, we~get
   \begin{align*}
\is{\bscr_2(\wlam) \wlam^m e_0}{\wlam^n e_0} & =
\sqrt{\hat \lambda_m \hat \lambda_n} \;
\is{\bscr_2(\wlam) e_m}{e_n}
   \\
& = \sqrt{\hat \lambda_m \hat \lambda_n} \; \beta_m
\is{e_m}{e_n}, \quad m,n \in \zbb_+.
   \end{align*}
This yields
   \begin{align} \label{zserw}
\is{\bscr_2(\wlam) \wlam^m e_0}{\wlam^n e_0} = 0,
\quad m,n \in \zbb_+, \, m \neq n.
   \end{align}

Denote by $\mu_1$ the nonzero finite Borel measure on
$\rbb_+$ defined by
   \begin{align} \label{muqes}
\mu_1(\varDelta) = \mu_0 (\varphi^{-1}(\varDelta)),
\quad \varDelta \in \borel{\rbb_+},
   \end{align}
where $\varphi\colon \rbb_+ \to \rbb_+$ is the
function given by $\varphi(x)=\sqrt{x}$ for $x\in
\rbb_+$. Let $\rho$ be the nonzero finite Borel
measure on $\cbb$ defined by
   \begin{align} \label{rygso}
\rho(\varDelta) = \frac{1}{2\pi}\int_{\rbb_+}
\int_0^{2\pi} \chi_{\varDelta}(t \E^{\I \theta}) \D
\theta \D \mu_1(t), \quad \varDelta \in \borel{\cbb},
   \end{align}
where $\chi_{\varDelta}$ stands for the characteristic
function of $\varDelta$. Note that $\rho$ is compactly
supported (because $\mu_0$ and $\mu_1$ are so) and
rotation-invariant. Applying the standard measure
theory approximation procedure to \eqref{rygso}, we
deduce that for every bounded complex Borel function
$f$ on $\cbb$,
   \begin{align} \label{rygtro}
\int_{\cbb} f \D \rho = \frac{1}{2\pi}\int_{\rbb_+}
\int_0^{2\pi} f(t \E^{\I \theta}) \D \theta \D
\mu_1(t).
   \end{align}
It follows from \eqref{yzldwa}, \eqref{zserw} and
\cite[Theorem~1.6.12]{Ash00} that \allowdisplaybreaks
   \begin{align*}
\is{\bscr_2(\wlam) \wlam^m e_0}{\wlam^n e_0} & =
\int_{\rbb_+} \sqrt{x}^{m+n} \D \mu_0(x) \delta_{m,n}
   \\
& = \int_{\rbb_+} t^{m+n} \D \mu_1(t) \delta_{m,n}
   \\
& = \int_{\rbb_+} t^{m+n} \D \mu_1(t) \frac{1}{2\pi}
\int_0^{2\pi} \E^{\I (m-n) \theta} \D \theta
   \\
& \hspace{-1.8ex}\overset{\eqref{rygtro}}= \int_{\cbb}
z^{m} \bar z^{n} \D \rho(z), \quad m,n \in \zbb_+.
   \end{align*}
Summarizing, we have proved that
   \begin{align} \label{bisz}
\is{\bscr_2(\wlam) \wlam^m e_0}{\wlam^n e_0} =
\int_{\cbb} z^{m} \bar z^{n} \D \rho(z), \quad m,n \in
\zbb_+.
   \end{align}

Let $\ee=\lin \{e_n\colon n \in \zbb_+\}$. It can be
easily seen that $\bar \ee=\ell^2$ and
   \begin{align} \label{euwq}
\ee= \Big\{p(\wlam) e_0\colon p \in \cbb[z]\Big\}.
   \end{align}
Since $\bscr_2(\wlam) \Ge 0$, the map $\ee \times \ee
\ni (f,g) \longmapsto \is{f}{g}_2=\is{\bscr_2(\wlam)
f}{g} \in \cbb$ is a semi-inner product on $\ee$.
Denote by $\|\cdot\|_2$ the corresponding seminorm,
that is,
   \begin{align} \label{trnym}
\|h\|_2=\sqrt{\is{\bscr_2(\wlam) h}{h}}, \quad h\in
\ee.
   \end{align}
According to \eqref{bisz}, we have
   \begin{align} \label{wlez}
\is{\wlam^m e_0}{\wlam^n e_0}_2 = \int_{\cbb} z^{m}
\bar z^{n} \D \rho(z), \quad m,n \in \zbb_+.
   \end{align}
It follows that
   \begin{align} \label{nyrw}
\|p(\wlam)e_0\|_2^2 = \int_{\cbb} |p|^2 \D \rho, \quad
p \in \cbb[z].
   \end{align}
Set $\nn = \{h\in \ee\colon \|h\|_2=0\}$. Since
$\|\cdot\|_2$ is seminorm, $\nn$ is vector subspace of
$\ee$. Denote by $(\widehat \hh,
\isii{\cdot}{\mbox{-}})$ the Hilbert space completion
of the inner product space $(\ee/\nn,
\isii{\cdot}{\mbox{-}})$, where
   \begin{align} \label{treil}
\isii{f+\nn}{g+\nn}=\is{f}{g}_2, \quad f,g \in \ee.
   \end{align}
The corresponding norm is denoted by
$\vertiii{\cdot}$. This, together with \eqref{wlez},
yields
   \begin{align} \label{dyszek}
\isii{p(\wlam)e_0 + \nn}{q(\wlam)e_0 + \nn} =
\int_{\cbb} p \bar q \D \rho, \quad p,q \in \cbb[z].
   \end{align}

The next step is to construct a unitary isomorphism $U
\in \ogr{\widehat \hh, H^2(\rho)}$ and an operator
$J\in \ogr{\ell^2,\widehat \hh}$, the product of which
will be the required operator $X$. First, note that by
\eqref{euwq}, we have
   \begin{align} \label{gynts}
\ee/\nn = \{p(\wlam) e_0 + \nn\colon p \in \cbb[z]\}.
   \end{align}
Since $\ee/\nn$ is dense in $\widehat \hh$, it follows
from \eqref{dyszek} that there exists a unique unitary
isomorphism $U \in \ogr{\widehat \hh, H^2(\rho)}$ such
that
   \begin{align} \label{upser}
U (p(\wlam) e_0 +\nn) = p \;\; \text{a.e.\ $[\rho]$},
\quad p \in \cbb[z].
   \end{align}
Next, observe that
   \begin{align*}
\vertiii{f+\nn}^2 \overset{\eqref{treil}} = \|f\|_2^2
\overset{\eqref{trnym}} = \is{\bscr_2(\wlam) f}{f} \Le
\|\bscr_2(\wlam)\|\|f\|^2, \quad f\in \ee.
   \end{align*}
Since $\bar \ee = \ell^2$, there exists a unique
operator $J\in \ogr{\ell^2,\widehat \hh}$ such that
   \begin{align*}
Jh = h + \nn, \quad h\in \ee.
   \end{align*}
Clearly, $\overline{\ob{J}}=\widehat{\hh}$. Define the
operator $X\in \ogr{\ell^2,H^2(\rho)}$ by $X=UJ$. Then
$\overline{\ob{X}}=H^2(\rho)$. Moreover, we have
   \begin{align*}
(X\wlam)p(\wlam)e_0 & = UJ q(\wlam)e_0 = U(q(\wlam)e_0
+ \nn)
   \\
& \hspace{-2.2ex}\overset{\eqref{upser}}= q = M_z p
\overset{\eqref{upser}}= M_z U (p(\wlam)e_0 + \nn)
   \\
& = M_z UJp(\wlam)e_0 = (M_z X) p(\wlam)e_0 \;\;
\text{a.e.\ $[\rho]$}, \quad p \in \cbb[z],
   \end{align*}
where $q(z)=zp(z)$. By \eqref{euwq} and
$\bar{\ee}=\ell^2$, we conclude that $X\wlam = M_z X$,
which proves (i).

(ii) Since $\bar \ee = \ell^2$, (ii) is a direct
consequence of the following computation:
   \begin{align*}
\is{X^*Xf}{g}=\is{J^*Jf}{g}=\isii{Jf}{Jg} & =
\isii{f+\nn}{g+\nn}
   \\
& \hspace{-2.2ex}\overset{\eqref{treil}}= \is{f}{g}_2
= \is{\bscr_2(\wlam)f}{g}, \quad f,g\in \ee.
   \end{align*}

(iii) This is a direct consequence of the following
more general fact.
   \begin{align} \label{wcdpo}
   \begin{minipage}{60ex}
{\em Let $\wlam$ be a CPD unilateral
weighted shift. Then $\wlam$ is of type~III
if and only if $\|\cdot\|_2$ is a norm, or
equivalently if and only if $X$ is
injective, or equivalently if and only if
$\beta_n > 0$ for every $n\in \zbb_+$.}
   \end{minipage}
   \end{align}
To prove \eqref{wcdpo}, first note that
   \begin{align} \label{bexa}
\forall p \in \cbb[z]\colon p(\wlam)e_0=0 \implies
p=0.
   \end{align}
Indeed, if $p(z)=\sum_{j=0}^n \alpha_j z^j$ for some
$n\in \zbb_+$ with $\{\alpha_j\}_{j=0}^{\infty}
\subseteq \cbb$, then
   \begin{align*}
0=p(\wlam)e_0 = \sum_{j=0}^n \alpha_j \wlam^j e_0 =
\sum_{j=0}^n \alpha_j \sqrt{\hat \lambda_j} \; e_j,
   \end{align*}
so by the linear independence of
$\{e_k\}_{k=0}^{\infty}$, $\alpha_j=0$ for every $j=0,
\ldots, n$, and thus $p=0$. Next, using \eqref{nyrw}
and \eqref{bexa}, we see that $\|\cdot\|_2$ is a norm
if and only if
   \begin{align*}
\forall p \in \cbb[z]\colon \int_{\cbb} |p|^2 \D
\rho=0 \implies p=0,
   \end{align*}
or equivalently, by continuity of polynomials, if and
only if
   \begin{align} \label{fthwa}
\forall p \in \cbb[z] \colon \big[\;p=0 \text{ on }
\supp{\rho}\,\big] \implies p=0.
   \end{align}
Now there are two possibilities. First, $\supp{\rho}
\subseteq \{0\}$. Then the nonzero polynomial $p(z)=z$
vanishes on $\supp{\rho}$, which means that
\eqref{fthwa} does not hold. As a consequence,
$\|\cdot\|_2$ is not a norm. The second possibility is
that $\supp{\rho} \nsubseteq \{0\}$, that is, there
exists $w\in \supp{\rho} \setminus \{0\}$. Since
$\rho$ is rotation-invariant, we get
   \begin{align*}
\E^{\I \varphi} \supp{\rho} \subseteq \supp{\rho},
\quad \varphi\in [0,2\pi),
   \end{align*}
so $w \cdot \tbb \subseteq \supp{\rho}$. Hence,
according to the fundamental theorem of algebra,
\eqref{fthwa} holds, which means that $\|\cdot\|_2$ is
a norm. Summarizing, we have proved that $\|\cdot\|_2$
is a norm if and only if $\supp{\rho} \nsubseteq
\{0\}$. In view of \eqref{mam0}, \eqref{muqes} and
\eqref{rygso}, the latter holds if and only if
$\supp{\nu} \nsubseteq \{0\}$ or $\cfrak \neq 0$, or
equivalently if and only if $\wlam$ is of type~III. In
turn, in view of $\bscr_2(\wlam)\Ge 0$ and
\eqref{trnym}, the seminorm $\|\cdot\|_2$ is a norm if
and only if $\bscr_2(\wlam)|_{\ee}$ is injective.
Therefore, by \eqref{zewq}, this is equivalent to
$\beta_n > 0$ for all $n\in \zbb_+$, so it is
equivalent to $\jd{\bscr_2(\wlam)}=\{0\}$ and, by
(ii), this is equivalent to $\jd{X}=\{0\}$, which
proves \eqref{wcdpo}.

(iv) Suppose that $\wlam$ is of type~III. Using
\eqref{mam0}, \eqref{muqes} and \eqref{rygso}, we
deduce that $\supp{\rho} \nsubseteq \{0\}$. This
implies that $\|\zeta^n\|^2=\int_{\cbb} |z|^{2n}\D
\rho(z) > 0$ for every $n\in \zbb_+$, where $\zeta$ is
the identity function on $\cbb$. Applying
\eqref{rygtro}, we verify that the sequence
$\Big\{\frac{\zeta^n}{\|\zeta^n\|}\Big\}_{n=0}^{\infty}$
is an orthonormal basis of the Hilbert space
$H^2(\rho)$. Since
   \begin{align*}
M_z \Big(\frac{\zeta^n}{\|\zeta^n\|}\Big) =
\frac{\|\zeta^{n+1}\|}{\|\zeta^n\|}
\Big(\frac{\zeta^{n+1}}{\|\zeta^{n+1}\|}\Big), \quad n
\in \zbb_+,
   \end{align*}
we see that $M_z$ is unitarily equivalent to
the unilateral weighted shift $W_{\omegab}$
with weights $\omegab :=
\Big\{\frac{\|\zeta^{n+1}\|}
{\|\zeta^n\|}\Big\}_{n=0}^{\infty}$ (so
$W_{\omegab}$ is subnormal). Hence, by
\cite[Theorem~1.6.12]{Ash00}, the
corresponding formal moment sequence of
$W_{\omegab}$ is given by (see
\eqref{mur-hupy}),
   \begin{align*}
\frac{\|\zeta^1\|^2}{\|\zeta^0\|^2}
\frac{\|\zeta^2\|^2}{\|\zeta^1\|^2} \cdots
\frac{\|\zeta^n\|^2}{\|\zeta^{n-1}\|^2} =
\frac{\|\zeta^n\|^2}{\|\zeta^0\|^2}&=
\frac{1}{\rho(\cbb)} \int_{\cbb} |z|^{2n}\D \rho(z)
   \\
&=\frac{1}{\rho(\cbb)} \int_{\rbb_+} x^n\D \rho\circ
\phi^{-1}(z) \quad n \Ge 1.
   \end{align*}
This, together with \eqref{Stieq}, implies
that $W_{\omegab}$ is a subnormal unilateral
weighted shift with the Berger measure
$\frac{1}{\rho(\cbb)} \rho\circ \phi^{-1}$,
which proves (iv).

(v) Since $X$ has dense range, we see that $X$ has a
bounded inverse if and only if there exists
$\varepsilon>0$ such that $\|Xh\|^2 \Ge \varepsilon
\|f\|^2$ for every $f\in \ell^2$. By (ii), this is
equivalent to $\bscr_2(\wlam) \Ge \varepsilon I$.
Since $\bscr_2(\wlam)$ satisfies \eqref{zewq}, we
conclude that $X$ has a bounded inverse if and only if
\eqref{bydqa} holds. Finally, if condition
\eqref{bydqa} is satisfied, then according to
\eqref{wcdpo}, $\wlam$ is of type~III, so applying (i)
and (iv) completes the proof.
   \end{proof}
The following corollary is an immediate consequence of
Theorem~\ref{mainth}, \cite[Theorem~1]{Cl75} and the
fact that subnormal operators are hyponormal (see
\cite[Proposition~II.4.2]{Con91}).
   \begin{cor}
Let $\wlam\in \ogr{\ell^2}$ be a CPD
unilateral weighted shift of type~III. Then
there exists a subnormal unilateral weighted
shift $W_{\omegab}\in \ogr{\ell^2}$ and a
quasi-invertible operator $X\in
\ogr{\ell^2}$ such that $X\wlam=W_{\omegab}
X$. If this is the case, then
$\sigma(W_{\omegab}) \subseteq
\sigma(\wlam)$.
   \end{cor}
The inclusion $\sigma(M_z) \subseteq \sigma(\wlam)$ is
valid even if $\wlam$ is not of type III, but then by
\eqref{dynqe}, $\sigma(M_z)=\emptyset$ if $\wlam$ is
of type I and $\card{\sigma(M_z)}=1$ if $\wlam$ is of
type II (recall that \cite[Theorem~1]{Cl75} is true
even if the intertwining operator $X$ is not
injective).
   \begin{rem}
Regarding the proof of Theorem~\ref{mainth}, we make
the following observation. We show that the operator
$R:=U^* M_z U$ can be described explicitly. For, take
any $h\in \ee$. Then, by \eqref{euwq}, $h=p(\wlam)e_0$
for some $p\in \cbb[z]$. Since $\rho$ is compactly
supported, $r:=\sup\{|z|\colon z\in \supp{\rho}\} <
\infty$. Hence, we have
   \allowdisplaybreaks
   \begin{align} \notag
\vertiii{(\wlam h) + \nn}^2 &= \vertiii{(\wlam
p(\wlam)e_0) + \nn}^2
   \\ \notag
& \hspace{-2.2ex} \overset{\eqref{treil}} = \|\wlam
p(\wlam)e_0\|_2^2
   \\ \notag
& \hspace{-2.2ex}\overset{\eqref{nyrw}}=
\int_{\supp{\rho}} |zp(z)|^2 \D \rho(z)
   \\ \notag
& \Le r^2 \int_{\supp{\rho}} |p(z)|^2 \D \rho(z)
   \\ \label{ofgip}
& \hspace{-2.2ex} \overset{\eqref{nyrw}} = r^2
\|p(\wlam)e_0\|_2^2 = r^2 \vertiii{h + \nn}^2.
   \end{align}
This implies that $\wlam(\nn) \subseteq \nn$ and the
operator $\widetilde R\colon \ee/\nn \to \ee/\nn$
given by
   \begin{align} \label{rhns}
\widetilde R(h + \nn) = (\wlam h) + \nn, \quad h\in
\ee,
   \end{align}
is well defined on the quotient vector space
$\ee/\nn$. By \eqref{ofgip}, $\widetilde R$ is bounded
with respect to the norm $\vertiii{\cdot}$. In order
to conserve notation, we continue to write $\widetilde
R$ for the unique extension of the original operator
$\widetilde R$ to a bounded linear operator on
$\widehat \hh$. Finally, applying \eqref{upser} and
\eqref{rhns}, we obtain
   \begin{align*}
(U\widetilde R)(p(\wlam) e_0 + \nn) & = U(q(\wlam)e_0
+ \nn)
   \\
& = q = M_z p = (M_z U) (p(\wlam) e_0 + \nn)\;\;
\text{a.e.\ $[\rho]$}, \quad p\in \cbb[z],
   \end{align*}
where $q(z)=zp(z)$. Using \eqref{gynts}, we conclude
that $\widetilde R=U^*M_z U$, so $\widetilde R=R$.
Therefore, the formula \eqref{rhns} explicitly
describes the operator $R$.
   \hfill $\diamondsuit$
   \end{rem}
   \subsection{Quasi-affine transform}
According to the main result of this paper,
Theorem~\ref{mainth}, any CPD unilateral
weighted shift of type III (see
Definition~\ref{kgsw}) is a quasi-affine
transform of a subnormal unilateral weighted
shift. However, this result does not answer
the question of whether a CPD unilateral
weighted shift, which is a quasi-affine
transform of a subnormal unilateral weighted
shift, is of type III. In this subsection,
we will answer this question in the
negative. Namely, we will prove, among other
things, that a non-subnormal CPD unilateral
weighted shift of type I or type II is a
quasi-affine transform of any contractive
subnormal unilateral weighted shift (see
Proposition~\ref{alevy}). Summarizing, this
means that any CPD unilateral weighted shift
is a quasi-affine transform of a subnormal
unilateral weighted shift.

We begin by showing that a CPD unilateral
weighted shift which is a quasi-affine
transform of a subnormal operator must be a
quasi-affine transform of a multiplication
operator by the independent variable on some
$H^2(\tau)$. However, we have no guarantee
that such a measure $\tau$ is
rotation-invariant (cf.\
Theorem~\ref{mainth}).
   \begin{lem} \label{incsq}
Let $\wlam \in \ogr{\ell^2}$ be a CPD
unilateral weighted shift, $S\in \ogr{\hh}$
be a subnormal operator, $X\in \ogr{\ell^2,
\hh}$ be an operator with dense range such
that $X \wlam = S X$. Then there exist a
compactly supported finite Borel measure
$\tau$ on $\cbb$ and an operator $\widehat
X\in \ogr{\ell^2,H^2(\tau)}$ with dense
range such that $\widehat X \wlam = M_z
\widehat X$ and $X^*X=\widehat X^* \widehat
X$. In particular, if $X$ is
quasi-invertible $($resp.\ invertible$)$,
then so is~$\widehat X$.
   \end{lem}
   \begin{proof}
Clearly $X\wlam^n e_0 = S^n Xe_0$ for all
$n\in \zbb_+$. Thus we can assume that
$f_0:=Xe_0 \neq 0$. Since $X$ has dense
range and $e_0$ is cyclic for $\wlam$, $f_0$
is cyclic for $S$. The following argument is
well know (see e.g., \cite[Subsections~9 and
14]{St-Sz89}). If $N\in \ogr{\kk}$ is a
normal extension of $S$ and $E$ is its
spectral measure, then
   \begin{align*}
\is{S^n f_0}{S^m f_0} = \is{N^n f_0}{N^m
f_0} = \int_{\cbb} \bar z^m z^n \D \tau(z),
\quad m,n \Ge 0,
   \end{align*}
where
$\tau(\varDelta)=\is{E(\varDelta)f_0}{f_0}$
for $\varDelta \in \borel{\cbb}$. Thus there
exists a unique unitary isomorphism $U\in
\ogr{\hh,H^2(\tau)}$ such that $U(p(S)
f_0)=p$ a.e.\ $[\tau]$ for every $p\in
\cbb[z]$. As in the proof of
Theorem~\ref{mainth}, we get $US=M_z U$. Set
$\widehat X= UX\in \ogr{\ell^2, H^2(\tau)}$.
Then $\widehat X$ has dense range, $\widehat
X \wlam = M_z \widehat X$ and $\widehat X^*
\widehat X = X^*U^*U X=X^*X$.
   \end{proof}
In contrast to Theorem~\ref{mainth}, the
following result shows that a single
non-subnormal CPD unilateral weighted shift
is a quasi-affine transform of a huge number
of subnormal unilateral weighted shifts.
   \begin{pro}\label{alevy}
A non-subnormal CPD unilateral weighted
shift is a quasi-affine transform of any
contractive subnormal unilateral weighted
shift $($including the unilateral
shift\/$)$. Moreover, a subnormal unilateral
weighted shift of type I or type II has norm
$1$.
   \end{pro}
   \begin{proof}
Let $\wlam\in \ogr{\ell^2}$ be a CPD
unilateral weighted shift with the
representing triplet $(\bfrak, \cfrak,
\nu)$.

Let us first consider the case where $\wlam$
is not subnormal. Set
   \allowdisplaybreaks
   \begin{align*}
\varOmega & = \bigcap_{n \in
\zbb_+}\Big\{t\in \rbb \colon 1 + t n +
\cfrak n^2 + \int_{\rbb_+} Q_n \D\nu > 0
\Big\},
      \\
\underline{\bfrak} & =\inf \varOmega,
   \\
\vartheta & = \sup{\supp{\nu}},
      \\
\varGamma_j & =\int_{\rbb_+}
\frac{1}{(1-x)^j} \D \nu(x) \text{ (provided
the integral exists)}, \quad j=1,2.
   \end{align*}
Recall that by \cite[Sec.~5]{Ja-Ju-St22},
$\underline{\bfrak} \in \rbb$. The
description of the set $\varOmega$ is given
in Table~1 and follows from
\cite[Theorem~5.2]{Ja-Ju-Le-St23}.
   \begin{table}
   \centering \caption{The description of
the set $\varOmega$.
   \\
\centering In cases \ding{172}, \ding{173}
and \ding{174}, $-\infty <
\underline{\bfrak} < 0$. In case \ding{175},
$-\varGamma_1 < \underline{\bfrak}< 0$.}
   \begin{tabular}{||c||c|c|c|c||c||}
   \hline Case & $\nu$ & $\cfrak$ &
$\varGamma_1$ & $\varGamma_2$ & $\varOmega$
   \\
   \hline \hline \ding{172} & $\vartheta >
1$ & & & & $(\underline{\bfrak},\infty)$
   \\
\hline \ding{173} & $\vartheta \Le 1$ &
$\cfrak > 0$ & & &
$(\underline{\bfrak},\infty)$
   \\
\hline \ding{174} & $\vartheta \Le 1$ &
$\cfrak = 0$ & $\varGamma_1 = \infty$ & &
$(\underline{\bfrak},\infty)$
   \\
\hline \ding{175} & $\vartheta \Le 1$ &
$\cfrak = 0$ & $\varGamma_1 < \infty$ &
$\varGamma_2 > 1$ &
$(\underline{\bfrak},\infty)$
   \\
\hline \ding{176} & $\vartheta \Le 1$ \&
$\nu=\delta_0$ & $\cfrak = 0$ & $\varGamma_1
< \infty$ & $\varGamma_2 =1$ &
$(-\varGamma_1,\infty)$
   \\
\hline \ding{177} & $\vartheta \Le 1$ \&
$\nu \neq \delta_0$ & $\cfrak = 0$ &
$\varGamma_1 < \infty$ & $\varGamma_2 =1$ &
$[-\varGamma_1,\infty)$
   \\
\hline \ding{178} & $\vartheta \Le 1$ &
$\cfrak = 0$ & $\varGamma_1 < \infty$ &
$\varGamma_2 < 1$ & $[-\varGamma_1,\infty)$
   \\
\hline
   \end{tabular}
   \end{table}
Note that under our assumptions, in cases
\ding{177} and \ding{178}, $\bfrak\in
(-\varGamma_1,\infty)$. Indeed, otherwise
$\varGamma_2\Le 1$, $\bfrak = - \varGamma_1$
and $\cfrak=0$, so in view of
Proposition~\ref{rewa}, $\wlam$ is
subnormal, a contradiction. Therefore,
according to Table~1 and
\cite[(5.6)]{Ja-Ju-Le-St23}, in cases
\ding{175}, \ding{176}, \ding{177} and
\ding{178}, $\lim_{n\to \infty} \hat
\lambda_n = \infty$. From Table~1 and the
proof of statement (i) of
\cite[Theorem~5.2]{Ja-Ju-Le-St23}, it
follows that in cases \ding{172}, \ding{173}
and \ding{174}, $\lim_{n\to \infty} \hat
\lambda_n = \infty$. As a consequence, in
all cases $\lim_{n\in \zbb_+} \hat \lambda_n
= \infty$.

Now let $W_{\omegab}\in \ogr{\ell^2}$ be any
contractive subnormal unilateral weighted
shift with weights
$\omegab=\{\omega_n\}_{n=0}^{\infty}$. Then
$\sup_{n\in \zbb_+} \omega_n =
\|W_{\omegab}\| \le 1$, which implies that
$\hat \omega_n \Le 1$ for every $n \in
\zbb_+$. Combined with $\lim_{n\in \zbb_+}
\hat \lambda_n = \infty$, this implies that
$\lim_{n \to \infty} \frac{\hat
\omega_n}{\hat \lambda_n} =0$. Thus, by
Lemma~\ref{wpwnc}, $\wlam$ is a quasi-affine
transform of $W_{\omegab}$.

Now suppose that $\wlam$ is subnormal. If
$\wlam$ is of type I, then by
\cite[Proposition~4.5]{Sh-At00}, $\wlam$ is
an isometry, so $\|\wlam\|=1$. If $\wlam$ is
of type II, then by Corollary~\ref{bcny},
$C=0$ and $\supp{F}=\{0\}$, where $(B,C,F)$
is the representing triplet of $\wlam$.
Therefore, by
\cite[Theorem~3.2.5(b)]{Ja-Ju-St22},
$\supp{M}=\{0\}$, so condition (i) of
\cite[Proposition~4.3.5]{Ja-Ju-St22} is
satisfied. This, together with assertion (b)
of \cite[Proposition~4.3.5]{Ja-Ju-St22},
implies that\footnote{In fact, we have
proved more, namely that if $T\in \ogr{\hh}$
is a nonzero subnormal operator with the
representing triplet $(B,C,F)$ such that
$\supp{F}\subseteq \{0\}$, then $\|T\|=1$.}
$\|\wlam\|=1$. This completes the proof.
   \end{proof}
   \begin{rem}
Note that if $W_{\omegab}\in \ogr{\ell^2}$
is a subnormal unilateral weighted shift
with the Berger measure $\mu$ such that
$\theta_1:=\inf{\supp{\mu}} > 1$ and $\wlam$
is a CPD unilateral weighted shift with the
scalar representing triplet $(\bfrak,
\cfrak, \nu)$ such that $1 <
\theta_2:=\sup{\supp{\nu}} < \theta_1$, then
$W_{\omegab}$ is a quasi-affine transform of
$\wlam$. Indeed, by \eqref{Stieq} and
\eqref{wnezero} we have
   \allowdisplaybreaks
   \begin{align*}
\frac{\hat \lambda_n}{\hat \omega_n} & =
\frac{1 + \bfrak n + \cfrak n^2 +
\int_{[0,\theta_2]} Q_n
\D\nu}{\int_{[\theta_1, \infty)} x^n \D
\mu(x)}
   \\
& \Le \frac{1 + \bfrak n + \cfrak n^2 +
\int_{[0,1)} Q_n \D\nu + \int_{[1,\theta_2]}
Q_n \D\nu}{\theta_1^n}
   \\
& \Le \frac{1 + \bfrak n + \cfrak
n^2}{\theta_1^n} + \frac{1}{\theta_1^n}
\nu(\rbb_+)\frac{(n-1)n}{2}(1 + \theta_2^n),
\quad n \in \zbb_+.
   \end{align*}
Since $1 < \theta_2 < \theta_1$, we see that
$\lim_{n\to \infty} \frac{\hat
\lambda_n}{\hat \omega_n} =0$. Therefore, by
Lemma~\ref{wpwnc}, $W_{\omegab}$ is a
quasi-affine transform of $\wlam$. Clearly,
if $\cfrak
> 0$, then $\wlam$ is not subnormal (see
Proposition~\ref{rewa} for other criteria).
   \hfill $\diamondsuit$
   \end{rem}
   \subsection{\label{Sec.2.3}Criteria for
type III} In this subsection we provide
criteria for determining whether a CPD
unilateral weighted shift is of type III,
which will consequently allow us to
establish criteria for determining whether
it is of type I or II. The criteria can be
found in Theorem~\ref{bacyt} (see also
Remark~\ref{wymem}). One of the criteria
refers to the seminorm $\|\cdot\|_2$, which
appears in the proof of
Theorem~\ref{mainth}. Another refers to the
unilateral weighted shift $W_{a,b}$ that was
discussed in
\cite[Example~4.3.6]{Ja-Ju-St22}. Let us
recall that $W_{a,b}\in \ogr{\ell^2}$ is the
unilateral weighted shift with weights
defined by
   \begin{align} \label{wysgi}
\lambda_n =
   \begin{cases}
\sqrt{a} & \text{ if } n=0,
   \\[1ex]
\sqrt{\frac{1 + n(b-1)}{1 + (n-1)(b-1)}} & \text{ if }
n \Ge 1,
   \end{cases}
   \end{align}
where $a \in (0,\infty)$ and $b\in [1,\infty)$. Note
that $a=\lambda_0^2$ and $b=\lambda_1^2$.

We will now give the aforementioned criteria for type
III.
   \begin{thm} \label{bacyt}
Suppose that $\wlam\in \ogr{\ell^2}$ is CPD. Let
$\|\cdot\|_2$ be as in \eqref{trnym} and
$\mcal=\bigvee_{j=1}^{\infty} \{e_j\}$. Then the
following conditions are equivalent{\em :}
   \begin{enumerate}
   \item[(i)] $\wlam|_{\mcal}$ is not a $2$-isometry,
   \item[(ii)] $\wlam$ is of type~III,
   \item[(iii)] $\wlam \neq W_{a,b}$ for all  $a \in
(0,\infty)$ and $b\in [1,\infty)$ such that
   \begin{align} \label{osqab}
1 - 2 a + a b \Ge 0,
   \end{align}
   \item[(iv)] $\|\cdot\|_2$ is a norm.
   \end{enumerate}
Moreover, if $\wlam = W_{a,b}$ for some $a \in
(0,\infty)$ and $b\in [1,\infty)$ satisfying
\eqref{osqab}, then $\wlam$ is of type~I if and only
if equality holds in~\eqref{osqab}.
   \end{thm}
   \begin{proof}
Recall that by \cite[Corollary~3.2.7(i)]{Ja-Ju-St22},
$\bscr_2(\wlam)\Ge 0$. Below we use the notation from
the proof of Theorem~\ref{mainth}.

(i)$\Rightarrow$(ii) Suppose, to the contrary, that
$\wlam$ is not of type~III, that is,
$\supp{\nu}\subseteq \{0\}$ and $\cfrak=0$. Then
according to \eqref{mam0}, $\supp{\mu_0}\subseteq
\{0\}$ and so $\supp{\rho}\subseteq \{0\}$. We
therefore have
   \begin{align} \notag
\|\sqrt{\bscr_2(\wlam)}p(\wlam)e_0\|^2 &
\overset{\eqref{trnym}} = \|p(\wlam)e_0\|_2^2
   \\ \label{gpezx}
& \overset{\eqref{nyrw}}= \int_{\cbb} |p|^2 \D \rho
=|p(0)|^2 \rho(\{0\}) =0, \quad p \in \cbb_0[z],
   \end{align}
where $\cbb_0[z]=\{p\in \cbb[z]\colon p(0)=0\}$. This
implies that $\bscr_2(\wlam)p(\wlam)e_0=0$ for every
$p\in \cbb_0[z]$, or equivalently,
$\bscr_2(\wlam)|_{\mcal}=0$. Thus we have
   \begin{align} \label{hmisw}
\|h\|^2 - 2 \|\wlam h\|^2 + \|\wlam^2 h\|^2 = 0, \quad
h\in \mcal,
   \end{align}
or equivalently, $\wlam|_{\mcal}$ is a $2$-isometry, a
contradiction.

(ii)$\Rightarrow$(i) Suppose, to the contrary, that
$\wlam|_{\mcal}$ is a $2$-isometry. Then identity
\eqref{hmisw} holds, so
   \begin{align*}
\sqrt{\bscr_2(\wlam)}p(\wlam)e_0=0, \quad p\in
\cbb_0[z].
   \end{align*}
Arguing as in \eqref{gpezx}, we deduce that
$\int_{\cbb} |p|^2 \D \rho=0$ for every $p\in
\cbb_0[z]$. Substituting $p(z)=z$, we get $\int_{\cbb}
|z|^2 \D \rho(z)=0$, which implies that $\supp{\rho}
\subseteq \{0\}$. As a consequence, $\supp{\mu_0}
\subseteq \{0\}$ and in view of \eqref{mam0},
$\supp{\nu} \subseteq \{0\}$ and $\cfrak=0$, which
contradicts the assumption that $\wlam$ is of
type~III.

(i)$\Rightarrow$(iii) This can be done by negating
(iii) and checking that then $\wlam|_{\mcal}$ is a
$2$-isometry (for this one can use
\cite[Lemma~6.1]{Ja-St01}).

(iii)$\Rightarrow$(i) Suppose, to the contrary, that
$\wlam|_{\mcal}$ is a $2$-isometry. Set
$a=\lambda_0^2$ and $b= \lambda_1^2$. Since
$\bscr(\wlam) \Ge 0$ (see
\cite[Corollary~3.2.7(i)]{Ja-Ju-St22}), we infer from
\eqref{zewq} that $1 - 2 a + a b = \beta_0 \Ge 0$.
Since $\wlam|_{\mcal}$ is a $2$-isometry, it follows
from \cite[Lemma~6.1]{Ja-St01} that $b\in [1,\infty)$
and $\wlam=W_{a,b}$, which contradicts (iii).

(ii)$\Leftrightarrow$(iv) Use \eqref{wcdpo}.

To prove the ``moreover'' part, note that if $\wlam$
is a $2$-isometry, then by \cite[Lemma~6.1]{Ja-St01},
$\beta_0=0$, or equivalently, equality holds
in~\eqref{osqab}. Conversely, if $\beta_0 = 0$, then,
as shown in \cite[Example~4.3.6]{Ja-Ju-St22},
$\bscr_2(\wlam)=0$, so $\wlam$ is a $2$-isometry.
   \end{proof}
The following is a direct consequence of
Theorem~\ref{bacyt}.
   \begin{cor} \label{mybtw}
Suppose that $\wlam\in \ogr{\ell^2}$ is CPD and is not
of type~I. Let $\|\cdot\|_2$ be as in \eqref{trnym}
and $\mcal=\bigvee_{j=1}^{\infty} \{e_j\}$. Then the
following are equivalent{\em :}
   \begin{enumerate}
   \item[(i)] $\wlam|_{\mcal}$ is not a $2$-isometry,
   \item[(ii)] $\wlam$ is of type~III,
   \item[(iii)] $\wlam \neq W_{a,b}$ for all $a \in (0,\infty)$ and $b\in
[1,\infty)$ satisfying $1 - 2 a + a b > 0$,
   \item[(iv)] $\|\cdot\|_2$ is a norm,
   \end{enumerate}
   \end{cor}
The answer to the question of when $W_{a,b}$
is subnormal is given below.
   \begin{cor}  \label{abWn}
Let $a \in (0,\infty)$ and $b\in
[1,\infty)$. Then $W_{a,b}$ is subnormal if
and only if $a \Le b=1$. If $a=b=1$, then
$W_{a,b}$ is an isometry, hence of type I.
If $a<b=1$, then $W_{a,b}$ is of type II.
Moreover, if $W_{a,b}$ is subnormal, then
$\|W_{a,b}\|=1$.
   \end{cor}
   \begin{proof}
Suppose that $W_{a,b}$ is subnormal. Since
$W_{a,b}$ is hyponormal, its weight sequence
is increasing (see \cite[Lemma, p.\
83]{shi74}), so $a\Le b$. Let
$\mcal=\bigvee_{j=1}^{\infty} \{e_j\}$. It
is clear that $W_{a,b}|_{\mcal}$ is
subnormal and, according to \eqref{wysgi}
and \cite[Lemma~6.1]{Ja-St01}, is unitarily
equivalent to the $2$-isometric unilateral
weighted shift $W_{\omegab}$ with weights
$\omegab :=\Big\{\sqrt{\frac{1 +
(n+1)(b-1)}{1 +
n(b-1)}}\Big\}_{n=0}^{\infty}$. Since
$W_{\omegab}$ is both subnormal and
$2$-isometric, $W_{\omegab}$ is isometric
(see, e.g., \cite[Theorem~3.4]{Ja-St01}), so
$b=1$. This yields $a \Le b=1$. Since
$\|W_{a,b}\|^2= \max\{a,b\}$ (see
\cite[Example~4.3.6]{Ja-Ju-St22}), we
conclude that $\|W_{a,b}\|= 1$ (cf.\
Proposition~\ref{alevy}).

Conversely, if $a \Le b=1$, then, by \eqref{wysgi},
the sequence $\lambdab=\{\lambda_n\}_{n=0}^{\infty}$
of weights of $W_{a,b}$ takes the form
   \begin{align} \label{cygseq}
\lambdab = (\sqrt{a}, 1, 1, 1, \ldots),
   \end{align}
so $\hat{\lambdab}=(1, a, a, a, \ldots)$, and
consequently
   \begin{align*}
\hat \lambda_n= \int_{\rbb_+} x^n \D \mu_a(x), \quad
n\in \zbb_+,
   \end{align*}
where $\mu_a:= (1-a)\delta_0 + a \delta_1$. By the
Berger theorem (see Theorem~\ref{Ber-G-W}), $W_{a,b}$
is a subnormal operator with the Berger measure
$\mu_a$.

If $a=b=1$, then, by \eqref{cygseq},
$W_{a,b}$ is an isometry, so it is of type
I. In turn, if $a<b=1$, then, $1 - 2 a + a b
= 1-a > 0$, so by Theorem~\ref{bacyt},
$W_{a,b}$ is of type II. This completes the
proof.
   \end{proof}
Suppose that $\wlam$ is CPD. It follows from
\eqref{wcdpo} that $\wlam$ is of type~III if and only
if the sequence $\{\beta_n\}_{n=0}^{\infty}$ defined
by \eqref{bydqa} has positive terms (recall that
$\beta_n \Ge 0$ for all $n\in \zbb_+$). We will now
show that the sequence $\{\beta_n\}_{n=0}^{\infty}$
satisfies the following dichotomy rule:
   \begin{align*}
   \begin{minipage}{53ex}
{\em Either all terms of $\{\beta_n\}_{n=0}^{\infty}$
are positive, or at most the zero term $\beta_0$ is
positive.}
   \end{minipage}
   \end{align*}
   \begin{pro} \label{kkra}
Suppose that $\wlam\in \ogr{\ell^2}$ is CPD and $k \in
\zbb_+$. Let $\{\beta_n\}_{n=0}^{\infty}$ be as in
\eqref{bydqa}. Then the following conditions are
equivalent{\em :}
   \begin{enumerate}
   \item[(i)] $\beta_k = 0$,
   \item[(ii)] $\wlam|_{\mcal_k}$ is a $2$-isometry,
where $\mcal_k:=\bigvee_{j=k}^{\infty} \{e_j\}$.
   \end{enumerate}
Moreover, if $\beta_k = 0$, then $\beta_j=0$ for every
integer $j\Ge 1$ and $\wlam=W_{a,b}$ with $a \in
(0,\infty)$ and $b\in [1,\infty)$ satisfying
\eqref{osqab}.
  \end{pro}
   \begin{proof}
Let $(B,C,F)$ be the representing triplet of $\wlam$
and $(\bfrak,\cfrak,\nu)$ be the scalar representing
triplet of $\wlam$. Clearly, the space $\mcal_k$ is
invariant for $\wlam$.

(i)$\Leftrightarrow$(ii) Consider first the case of
$k=0$. Then by \eqref{klaud} and \eqref{wnezero} we
have
   \begin{align*}
\beta_0 = 1 - 2\hat\lambda_1 + \hat\lambda_2= 2\cfrak
+ \nu(\rbb_+).
   \end{align*}
Therefore, $\beta_0=0$ if and only if $\cfrak=0$ and
$\nu=0$, or equivalently, by Corollary~\ref{bcny}, if
and if $C=0$ and $F=0$. Hence, according to
Lemma~\ref{dqaer}(ii), $\beta_0=0$ if and only if
$\wlam=\wlam|_{\mcal_0}$ is a $2$-isometry.

The case of arbitrary $k\in \zbb_+$ reduces to $k=0$,
because $\wlam|_{\mcal_k}$ is unitarily equivalent to
the CPD unilateral weighted shift
$W_{\boldsymbol{\lambda_k}}$ with weights
$\boldsymbol{\lambda_k}:=\{\lambda_{n+k}\}_{n=0}^{\infty}$
(relative to the orthonormal basis
$\{e_{n+k}\}_{n=0}^{\infty}$ of $\mcal_k$) and
$\beta_0(W_{\boldsymbol{\lambda_k}}) =
\beta_k(\wlam)$.

It remains to prove the ``moreover'' part. There are
two possibilities. First, the seminorm $\|\cdot\|_2$
is a norm. By \eqref{wcdpo}, we deduce that $\beta_n >
0$ for all $n\in \zbb_+$. Second, the seminorm
$\|\cdot\|_2$ is not a norm. Then by
Theorem~\ref{bacyt}, $\wlam=W_{a,b}$ with $\beta_0=1 -
2 a + a b \Ge 0$, so by
\cite[Example~4.3.6]{Ja-Ju-St22}, $\beta_j=0$ for
every $j\Ge 1$. This completes the proof.
   \end{proof}
   At this point, it is worth making the following
comments related to Theorems~\ref{mainth} and
\ref{bacyt}. Below, if $\wlam$ is CPD, we write
$(B,C,F)$ for the representing triplet of $\wlam$ and
$(\bfrak, \cfrak, \nu)$ for the scalar representing
triplet of $\wlam$.
   \begin{rem} \label{wymem}
(a) Suppose that $\wlam \in \ogr{\ell^2}$ is CPD. If
$\wlam$ is of type~I, then $\bscr_2(\wlam)=0$, so by
\eqref{trnym}, $\|\cdot\|_2=0$ and consequently $\dim
\widehat{\hh}=0$. Moreover, by Lemma~\ref{dqaer}(ii),
$F=0$ and $C=0$. If $\wlam$ is not of type~I and the
seminorm $\|\cdot\|_2$ is not a norm, then by
Corollary~\ref{mybtw}, $\wlam=W_{a,b}$ with $\theta:=1
- 2 a + a b > 0$. By Theorem~\ref{bacyt}, $\wlam$ is
of type~II. Since $\bscr_2(\wlam)$ is the diagonal
operator (relative to the orthonormal basis
$\{e_n\}_{n=0}^{\infty}$) with diagonal elements
$(\theta,0,0,\ldots)$, we see that
   \begin{align*}
\|h\|_2^2 = \is{\bscr_2(\wlam)h}{h} = \theta
|\is{h}{e_0}|^2, \quad h\in \ee.
   \end{align*}
This implies that $\nn=\lin\{e_j\colon j\in
\nbb\}$, and thus $\dim \widehat{\hh}=1$.
Finally, if $\|\cdot\|_2$ is a norm, then
$\dim \widehat{\hh}=\aleph_0$ and, by
\eqref{wcdpo}, $\wlam$ is of type~III.
Therefore, we get the following
   \begin{align} \label{dynqe}
\dim \widehat{\hh} =
   \begin{cases}
0 & \text{\em if $\wlam$ is of type~I,}
   \\
1 & \text{\em if $\wlam$ is of type~II,}
   \\
\aleph_0 & \text{\em if $\wlam$ is of
type~III.}
   \end{cases}
   \end{align}

(b) Let us assume that $\wlam$ is CPD and of type~II.
Clearly, $\wlam$ is not of type~I. In view of
Corollary~\ref{mybtw}, $\wlam=W_{a,b}$ with $\theta:=1
- 2 a + a b > 0$. Straightforward computations based
on \eqref{wysgi} show that
   \begin{align*}
\hat\lambda_n = a + (n-1)a(b-1), \quad n\Ge 1.
   \end{align*}
In turn, by \eqref{rnx-1} and \eqref{wnezero}, we have
   \begin{align*}
\hat\lambda_n = (1+\bfrak)+ (n-1)(\bfrak+\nu(\{0\})),
\quad n\Ge 1.
   \end{align*}
This shows that $\bfrak = a-1$, $\cfrak =0$ and
$\nu(\{0\})=\theta$. As a consequence, $\nu=\theta \,
\delta_0$. Since $\nu$ has an atom at $0$, we deduce
from \cite[Theorem~7.2]{Ja-Ju-Le-St23} that $\wlam$
has no $1$-step backward extension.

(c) Assume that $\wlam \in \ogr{\ell^2}$ is
such that $\wlam|_{\mcal}$ is a
$2$-isometry, where
$\mcal=\bigvee_{j=1}^{\infty} \{e_j\}$. Set
$a=\lambda_0^2$ and $b=\lambda_1^2$.
Clearly, $\wlam|_{\mcal}$ is unitarily
equivalent to the unilateral weighted shift
$W_{\omegab}$ with weights $\omegab =
\{\omega_n\}_{n=0}^{\infty}$, where
$\omega_n := \lambda_{n+1}$ for $n\in
\zbb_+$. Then, by Lemma~\ref{dqaer}(ii), the
operator $W_{\omegab}$ is a CPD unilateral
weighted shift for which $\tilde{\cfrak}=0$
and $\tilde{\nu}=0$, where $(\tilde{\bfrak},
\tilde{\cfrak}, \tilde{\nu})$ stands for the
scalar representing triplet of
$W_{\omegab}$. Applying
\cite[Lemma~6.1(ii)]{Ja-St01} to
$W_{\omegab}$, we deduce that
   \begin{align*}
\widehat{\omega}_n
\overset{\eqref{mur-hupy}}= 1 + n
(\lambda_1^2-1), \quad n\in \zbb_+.
   \end{align*}
This, together with Theorem~\ref{cpdws}
applied to $W_{\omegab}$, implies that
$\tilde{\bfrak}=b-1$. Note that the
unilateral weighted shift $\wlam$ is the
$1$-step backward extension of
$\wlam|_{\mcal}$. It follows from
\cite[Theorem~7.2]{Ja-Ju-Le-St23} that
$\wlam$ is CPD if and only if
   \begin{align*}
\frac{1}{a} \Ge \int_{\rbb_+} \frac{1}{x} \D
\tilde\nu(x) + 1 + \tilde{\cfrak} - \tilde{\bfrak} = 2
- b,
   \end{align*}
or equivalently that $\theta:=1 - 2 a + a b \Ge 0$.
   \hfill $\diamondsuit$
   \end{rem}
   \section{Examples}
   \subsection{\label{Sec.3.1}Examples via the measure
$\nu$} In this subsection we give sufficient
conditions, written only in terms of the measure
$\nu$, for a CPD unilateral weighted shift $\wlam$ to
be similar to a subnormal operator. We begin with a
preparatory lemma.
   \begin{lem}  \label{nyttrs}
Let $\wlam\in \ogr{\ell^2}$ be a CPD
unilateral weighted shift with the scalar
representing triplet $(\bfrak,\cfrak,\nu)$
satisfying the following conditions{\em :}
   \begin{enumerate}
   \item[(i)] $\vartheta:=\sup\supp{\nu} > 1$,
   \item[(ii)] $\int_{(1,\vartheta]}\frac{1}{(x-1)^2}\D \nu(x) <
\infty$,
   \item[(iii)] there exists a sequence
$\{\varepsilon_n\}_{n=1}^{\infty} \subseteq
(0,\infty)$ such that
   \begin{align*}
\liminf_{n\to\infty} (\vartheta-\varepsilon_n)
> 1 \; \text{and} \; \liminf_{n\to\infty}\nu([\vartheta-\varepsilon_n,\vartheta])\Big(1-
\frac{\varepsilon_n}{\vartheta}\Big)^n > 0.
   \end{align*}
   \end{enumerate}
Then $\wlam$ is of type III and is similar
to a subnormal unilateral weighted shift.
Moreover, if at least one of the following
conditions holds\footnote{Using (i), (ii)
and the Cauchy-Schwarz inequality, we see
that $\int_{(1,\vartheta]} \frac{1}{x-1} \D
\nu(x) < \infty,$ so $-\infty \Le
\int_{\rbb_+} \frac{1}{x-1} \D \nu(x) <
\infty$. Therefore, the ``moreover'' part
holds assuming only (i) and (ii).}{\em :}
   \begin{enumerate}
   \item[(a)]
$1 - \int_{(1,\vartheta]}\frac{1}{(x-1)^2}\D \nu(x) <
\int_{[0,1)} \frac{1}{(x-1)^2} \D \nu(x) \Le \infty$,
   \item[(b)] $\bfrak \neq \int_{\rbb_+}
\frac{1}{x-1} \D \nu(x)$,
   \item[(c)] $\cfrak > 0$,
   \end{enumerate}
then $\wlam$ is not subnormal.
   \end{lem}
   \begin{proof}
Note that
   \begin{align} \notag
\beta_n & \overset{\eqref{bydqa}}= 1-2\lambda_n^2 +
\lambda_n^2 \lambda_{n+1}^2
   \\ \notag
& \hspace{-.5ex}\overset{\eqref{mur-hupy}} =
\frac{\hat\lambda_n-2 \hat\lambda_{n+1} +
\hat\lambda_{n+2}}{\hat\lambda_n}
   \\  \notag
& \hspace{-.5ex} \overset{\eqref{wnezero}}=
\frac{2\cfrak + \int_{\rbb_+}(Q_n - 2 Q_{n+1} +
Q_{n+2})\D \nu}{\hat\lambda_n}
   \\ \label{estwq}
& \overset{\eqref{del2}}= \frac{2\cfrak +
\int_{\rbb_+} x^n \D \nu(x)}{1 + \bfrak n + \cfrak n^2
+ \int_{\rbb_+} Q_n \D\nu}, \quad n \in \zbb_+.
   \end{align}
Our next goal is to estimate the expression
$\int_{\rbb_+} Q_n \D\nu$ from below and above. We
begin with the estimation from above. By \eqref{klaud}
and \eqref{rnx-1}, we have
   \begin{align*}
Q_n(x) \Le
   \begin{cases}
\frac{(n-1)n}{2} & \text{ if $x\in [0,1]$},
   \\[1ex]
\frac{\vartheta^n}{(x-1)^2} & \text{ if $x\in
(1,\vartheta]$,}
   \end{cases}
\quad n\in \zbb_+,
   \end{align*}
which yields (recall that $\nu(\{1\})=0$)
   \begin{align} \label{qwas}
\int_{\rbb_+} Q_n \D \nu \Le \nu([0,1))
\frac{(n-1)n}{2} +
\int_{(1,\vartheta]}\frac{1}{(x-1)^2}\D \nu(x)\,
\vartheta^n , \quad n\in \zbb_+.
   \end{align}
Now, we proceed with the estimation from below:
   \begin{align} \label{iqnf}
\int_{\rbb_+} Q_n\D \nu \Ge \int_{(1, \vartheta]}
Q_n\D \nu \overset{\eqref{klaud}}\Ge \nu((1,
\vartheta]) \frac{(n-1)n}{2}, \quad n\in \zbb_+.
   \end{align}
Since, by (i), $\vartheta>1$ and $\vartheta \in
\supp{\nu}$, we deduce that $\nu((1, \vartheta])>0$
and consequently
$\int_{(1,\vartheta]}\frac{1}{(x-1)^2}\D \nu(x)
>0$. In particular, by  \eqref{iqnf}, $\int_{\rbb_+} Q_n\D \nu > 0$ for all
$n\Ge 2$ and
   \begin{multline*}
\frac{1 + \bfrak n + \cfrak n^2}{\int Q_n \D\nu} \Le
\frac{2(\frac{1}{n^2} + \frac{\bfrak}{n} +
\cfrak)}{\nu((1, \vartheta])(1-\frac{1}{n})} \text{
for all $n\Ge 2$ such that $1 + \bfrak n + \cfrak n^2
\Ge 0$.}
   \end{multline*}
This implies that
   \begin{align*}
\frac{1 + \bfrak n + \cfrak n^2}{\int Q_n \D\nu} & \Le
\frac{2 \cfrak}{\nu((1, \vartheta])} + 1 \text{ for
$n$ large enough,}
   \end{align*}
so
   \begin{align} \label{gfwqa}
0 < \frac{1 + \bfrak n + \cfrak n^2}{\int Q_n \D\nu} +
1 \Le \frac{2 \cfrak}{\nu((1, \vartheta])} + 2 \text{
for $n$ large enough.}
   \end{align}
By (iii), $\vartheta - \varepsilon_n
> 1$ for $n$ large enough, so
   \begin{align*}
\int_{\rbb_+} x^n \D\nu(x) \Ge \int_{[\vartheta -
\varepsilon_n,\vartheta]} x^n \D\nu(x) \Ge
\nu([\vartheta - \varepsilon_n,\vartheta])(\vartheta -
\varepsilon_n)^n \; \text{for $n$ large enough}.
   \end{align*}
This, together with (ii) and \eqref{qwas}, implies
that
   \begin{align*}
\frac{\int x^n \D \nu(x)}{\int Q_n \D\nu} & \Ge
\frac{\nu([\vartheta - \varepsilon_n,\vartheta])
\big(1 - \frac{\varepsilon_n}{\vartheta}\big)^n
}{\nu([0,1)) \frac{(n-1)n}{2\vartheta^n}+
\int_{(1,\vartheta]}\frac{1}{(x-1)^2}\D \nu(x)} \;
\text{for $n$ large enough}.
   \end{align*}
Therefore, by (i)-(iii), we have
   \begin{align*}
\liminf_{n\to \infty} \frac{\int x^n \D \nu(x)}{\int
Q_n \D\nu} > 0.
   \end{align*}
It follows that
   \begin{align}  \label{lynif}
\liminf_{n\to \infty} \beta_n &
\overset{\eqref{estwq}}= \liminf_{n\to \infty}
\frac{\frac{2\cfrak}{\int Q_n \D\nu} + \frac{\int x^n
\D \nu(x)}{\int Q_n
\D\nu}}{\underset{>0}{\underbrace{\frac{1 + \bfrak n +
\cfrak n^2}{\int Q_n \D\nu} + 1}}}
\overset{\eqref{gfwqa}}\Ge \liminf_{n\to \infty}
\frac{\frac{\int x^n \D \nu(x)}{\int Q_n
\D\nu}}{\frac{2 \cfrak}{\nu((1, \vartheta])} + 2} > 0.
   \end{align}
In view of (i), $\wlam$ is of type III.
Thus, by \eqref{wcdpo}, $\beta_n > 0$ for
every $n\in \zbb_+$. This, together with
\eqref{lynif}, shows that $\inf_{n\in
\zbb_+} \beta_n > 0$. Using
Theorem~\ref{mainth}(v), we conclude that
$\wlam$ is similar to a subnormal unilateral
weighted shift. Finally, if at least one of
conditions (a)-(c) holds, then, by
Proposition~\ref{rewa}, $\wlam$ is not
subnormal (for this, we only need (i) and
(ii)). This completes the proof.
   \end{proof}
In the case of the presence of an atom at the right
endpoint of the support of the measure $\nu$, the
question of similarity can be settled as follows.
   \begin{thm}\label{kdwq}
Let $\wlam\in \ogr{\ell^2}$ be a CPD
unilateral weighted shift with the scalar
representing triplet $(\bfrak,\cfrak,\nu)$
satisfying the following conditions{\em :}
   \begin{enumerate}
   \item[(i)] $\vartheta :=\sup\supp{\nu} > 1$,
   \item[(ii)] $\int_{(1,\vartheta]}\frac{1}{(x-1)^2}\D \nu(x) <
\infty$,
   \item[(iii)] $\nu(\{\vartheta\}) > 0$.
   \end{enumerate}
Then $\wlam$ is of type III and is similar
to a subnormal unilateral weighted shift.
Moreover, if at least one of conditions {\em
(a)}-{\em (c)} of Lemma~{\em \ref{nyttrs}}
is valid, then $\wlam$ is not subnormal.
   \end{thm}
   \begin{proof}
In view of Lemma~\ref{nyttrs}, it suffices to show
that condition (iii) of this lemma is valid. Set
$\varepsilon_n = \frac{1}{n}$ for $n\Ge 1$. By
continuity from above of finite measures, we have
   \begin{align*}
\lim_{n\to\infty}
\nu([\vartheta-\varepsilon_n,\vartheta])=\nu(\{\vartheta\}).
   \end{align*}
Using this equality and Euler's formula
   \begin{align*}
\lim_{n\to \infty} \Big(1+\frac{x}{n}\Big)^n= \E^{x},
\quad x \in \rbb,
   \end{align*}
we verify that
   \begin{align*}
\lim_{n\to\infty}\nu([\vartheta-\varepsilon_n,\vartheta])\Big(1-
\frac{\varepsilon_n}{\vartheta}\Big)^n =
\nu(\{\vartheta\}) \E^{-1/\vartheta}.
   \end{align*}
Hence, by (iii), condition (iii) of Lemma~\ref{nyttrs}
is satisfied.
   \end{proof}
Regarding Theorem~\ref{kdwq}, observe that if
$\vartheta > 1$ and $\nu((1,\eta))=0$ for some $1 <
\eta \Le \vartheta$, then condition (ii) of this
theorem is satisfied automatically. In particular,
this is the case if any of the conditions (i)-(iii) of
Theorem~\ref{ineqsuf} is valid. Consequently,
condition (iii) of Theorem~\ref{kdwq} (or
Lemma~\ref{nyttrs}) is not necessary for $\wlam$ to be
similar to a subnormal operator (see
Example~\ref{3uwre}).
   \subsection{\label{Sec.3.2}Examples via the scalar
representing triplet $(\bfrak,\cfrak, \nu)$}
   We begin with a lemma offering sufficient
conditions for similarity of a CPD unilateral weighted
shift to a subnormal operator written in terms of its
weights.
   \begin{lem}  \label{dfawyr}
Let $\wlam\in \ogr{\ell^2}$ be a CPD
unilateral weighted shift. Suppose that at
least one of the following conditions
holds{\em :}
   \begin{enumerate}
   \item[(i)] there exist $\tau\in (0,1)$ and $M\in
(0,\infty)$ such that $(1-\tau) (1+M)<1$ and $1 + \tau
\Le \lambda_n^2 \Le 1 + M$ for $n$ large enough,
   \item[(ii)]there exists $\tau\in [1,\infty)$
such that $1 + \tau \Le \lambda_n^2$ for $n$ large
enough.
   \end{enumerate}
Then $\wlam$ is of type III and is similar to a
subnormal unilateral weighted shift.
   \end{lem}
   \begin{proof}
(i) Note that
   \begin{align*}
\beta_n = 1 -2 \lambda_n^2 + \lambda_n^2
\lambda_{n+1}^2 & \Ge 1 -2 \lambda_n^2 +
(1+\tau)\lambda_n^2 = 1 - (1-\tau) \lambda_n^2
   \\
& \Ge 1 - (1-\tau) (1+M) > 0 \quad \text{ for $n$
large enough}.
   \end{align*}
Using Proposition~\ref{kkra} and the fact that
$\inf_{n\in \zbb_+} \beta_n \Ge 0$ (see \eqref{zewq}),
we conclude that $\inf_{n\in \zbb_+} \beta_n > 0$.
Hence, $\wlam$ is not of type I. According to
Theorem~\ref{mainth}(v), $\wlam$ is of type~III and is
similar to a subnormal unilateral weighted shift.

(ii) Likewise, we get
   \begin{align*}
\beta_n \Ge 1 +(\tau-1) \lambda_n^2 \Ge 1 + (\tau-1)
(\tau+1)=\tau^2 >0 \quad \text{ for $n$ large enough}.
   \end{align*}
Arguing as above, we come to the same conclusion.
   \end{proof}
In contrary to Theorem~\ref{kdwq}, the
similarity criterion given below appeals
largely to the parameters $\bfrak$ and
$\cfrak$ of the scalar representing triplet
$(\bfrak, \cfrak, \nu)$, with $\nu$
satisfying the mild requirements for the
endpoints of $\supp{\nu}$ and the total mass
of~$\nu$. To avoid over-complexity of the
problem studied in this paper, we will focus
on the case when $\bfrak\Ge 0$.
   \begin{thm} \label{ineqsuf}
Let $\wlam\in \ogr{\ell^2}$ be a CPD
unilateral weighted shift and
$(\bfrak,\cfrak, \nu)$ be its scalar
representing triplet with $\bfrak \Ge 0$
satisfying at least one of the following
conditions{\em :}
   \begin{enumerate}
   \item[(i)] $\bfrak + \cfrak \Ge 1$,
$2\cfrak + \nu(\rbb_+)(1 - \cfrak) \Ge \bfrak$,
$\nu(\rbb_+)\Ge 1$ and $\inf{\supp{\nu}} \Ge
2(1+\cfrak)$,
   \item[(ii)]  there exists $t \in (0,
\infty)$ such that $\bfrak + \cfrak + \nu(\rbb_+)t (1
+ t) \Ge 1$, \break $2\cfrak + \nu(\rbb_+)(1 - 2t
-\frac{3}{2}t^2) \Ge \bfrak$, $\nu(\rbb_+)t(2+t) \Ge 2
\cfrak$ and $\inf{\supp{\nu}} \Ge 2+t$,
   \item[(iii)] there exist  $t
\in (0, \infty)$ and $\tau \in (0, 1)$ such that
   \begin{itemize}
   \item[(iii-a)] $\bfrak + \cfrak
\Ge \tau$, $2\cfrak + \nu(\rbb_+)(1 - \frac{t}{2}) \Ge
\tau \bfrak$ and $\nu(\rbb_+) t \Ge 2 \tau \cfrak$,
   \item[(iii-b)] $\bfrak + \cfrak \Le \vartheta -1$,
$2\cfrak + \nu(\rbb_+) \Le (\vartheta -1) \bfrak$,
$(1-\tau)\vartheta < 1$ and $\inf{\supp{\nu}} \Ge
1+\tau + t$, where $\vartheta:=\sup\supp \nu$.
   \end{itemize}
   \end{enumerate}
Then $\wlam$ is of type III and is similar
to a subnormal unilateral weighted shift.
Moreover, if at least one of conditions {\em
(a)}-{\em (c)} of Lemma~{\em \ref{nyttrs}}
is valid, then $\wlam$ is not subnormal.
   \end{thm}
   \begin{proof}
(i) By \eqref{klaud}, we have
   \begin{align} \label{qkx}
Q_n(x) \Ge \frac{n(n-1)}{2}, \quad x\in [1,\infty), \,
n\in \zbb_+.
   \end{align}
From \eqref{rnx-0}, \eqref{wnezero} and the assumption
$\inf\sup \nu \Ge 2(1+\cfrak)$ it follows that
   \allowdisplaybreaks
   \begin{align*}
\hat\lambda_{n+1} & = (1 + \bfrak n + \cfrak n^2) +
(\bfrak + \cfrak) + \big(2\cfrak+\nu(\rbb_+)\big)n +
\int_{\rbb_+} x Q_n(x) \D\nu(x)
   \\
& \Ge (1 + \bfrak n + \cfrak n^2) + (\bfrak + \cfrak)
+ \big(2\cfrak+\nu(\rbb_+)\big)n + 2(1+\cfrak)
\int_{\rbb_+} Q_n \D\nu
   \\
& = \hat\lambda_n + (\bfrak + \cfrak) +
\big(2\cfrak+\nu(\rbb_+)\big)n + 2\cfrak \int_{\rbb_+}
Q_n \D\nu + \int_{\rbb_+} Q_n \D\nu
   \\
& \hspace{-1.8ex}\overset{\eqref{qkx}} \Ge
\hat\lambda_n + (\bfrak + \cfrak) +
\big(2\cfrak+\nu(\rbb_+)(1 - \cfrak)\big)n +
\nu(\rbb_+) \cfrak n^2 + \int_{\rbb_+} Q_n \D\nu
   \\
& \overset{\mathrm{(i)}} \Ge 2 \hat\lambda_n, \quad
n\in \zbb_+.
   \end{align*}
According to \eqref{mur-hupy},
$\lambda_n^2=\frac{\hat\lambda_{n+1}} {\hat\lambda_n}
\Ge 2$ for all $n\in \zbb_+$, so by
Lemma~\ref{dfawyr}(ii) with $\tau=1$, $\wlam$ is of
type III and is similar to a subnormal unilateral
weighted shift.

(ii) According to \eqref{klaud}, we have
   \allowdisplaybreaks
   \begin{align} \notag
Q_n(x) & = (n-1) + \sum_{j=1}^{n-2} (n-j-1) x^j
   \\ \notag
& \Ge (n-1) + \delta (1 + \ldots (n-2))
   \\ \label{nyrq}
& = (n-1) + \delta \frac{(n-1)(n-2))}{2}, \quad x\Ge
\delta \Ge 1, \, n \Ge 3.
   \end{align}
This, together with \eqref{rnx-0}, \eqref{wnezero} and
\eqref{nyrq} with $\delta=2+t$, implies that
   \begin{align*}
\hat\lambda_{n+1} & = (1 + \bfrak n + \cfrak n^2) +
(\bfrak + \cfrak) + \big(2\cfrak+\nu(\rbb_+)\big)n +
\int_{\rbb_+} x Q_n(x) \D\nu(x)
   \\
& \Ge (1 + \bfrak n + \cfrak n^2) + (\bfrak + \cfrak)
+ \big(2\cfrak+\nu(\rbb_+)\big)n + (2+t)\int_{\rbb_+}
Q_n \D\nu
   \\
& = \hat\lambda_n + (\bfrak + \cfrak) +
\big(2\cfrak+\nu(\rbb_+)\big)n + (1+t)\int_{\rbb_+}
Q_n \D\nu
   \\
& \Ge \hat\lambda_n + (\bfrak + \cfrak) +
\big(2\cfrak+\nu(\rbb_+)\big)n + \nu(\rbb_+) t (n-1)
   \\
& \hspace{15.3ex}+ \nu(\rbb_+) t(2+t)
\frac{(n-1)(n-2)}{2} + \int_{\rbb_+} Q_n \D\nu
   \\
& = \hat\lambda_n + \big(\bfrak + \cfrak + \nu(\rbb_+)
t (1 + t)\big) + \big(2\cfrak + \nu(\rbb_+)(1 - 2t
-\frac{3}{2}t^2)\big) n
   \\
& \hspace{15.3ex}+ \frac{\nu(\rbb_+)t(2+t)}{2} n^2 +
\int_{\rbb_+} Q_n \D\nu
   \\
& \hspace{-.5ex}\overset{\mathrm{(ii)}}\Ge 2
\hat\lambda_n, \quad n \Ge 3.
   \end{align*}
Therefore, $\lambda_n^2=\frac{\hat\lambda_{n+1}}
{\hat\lambda_n} \Ge 2$ for all integers $n\Ge 3$, so
by Lemma~\ref{dfawyr}(ii) with $\tau=1$, $\wlam$ is of
type III and is similar to a subnormal unilateral
weighted shift.

(iii) Arguing as above, we get
   \allowdisplaybreaks
   \begin{align*}
\hat\lambda_{n+1} & = (1 + \bfrak n + \cfrak n^2) +
(\bfrak + \cfrak) + \big(2\cfrak+\nu(\rbb_+)\big)n +
\int_{\rbb_+} x Q_n(x) \D\nu(x)
   \\
& \hspace{-1.5ex} \overset{\textrm{(iii-b)}}\Ge (1 +
\bfrak n + \cfrak n^2) + (\bfrak + \cfrak) +
\big(2\cfrak+\nu(\rbb_+)\big)n + (1+\tau + t)
\int_{\rbb_+} Q_n \D\nu
   \\
& = \hat\lambda_n + (\bfrak + \cfrak) +
\big(2\cfrak+\nu(\rbb_+)\big)n + t \int_{\rbb_+} Q_n
\D\nu + \tau \int_{\rbb_+} Q_n \D\nu
   \\
& \hspace{-1.8ex}\overset{\eqref{qkx}} \Ge
\hat\lambda_n + (\bfrak + \cfrak) + \big(2\cfrak +
\nu(\rbb_+)(1 - \frac{t}{2})\big)n + \nu(\rbb_+)
\frac{t}{2} n^2 + \tau \int_{\rbb_+} Q_n \D\nu
   \\
& \hspace{-1.5ex} \overset{\textrm{(iii-a)}}\Ge
(1+\tau) \hat\lambda_n, \quad n \in \zbb_+.
   \end{align*}
This together with \eqref{mur-hupy} implies that
   \begin{align} \label{tykwe}
\lambda_n^2=\frac{\hat\lambda_{n+1}} {\hat\lambda_n}
\Ge 1 + \tau, \quad n\in \zbb_+,
   \end{align}
Noting that
   \allowdisplaybreaks
   \begin{align*}
\hat\lambda_{n+1} & = (1 + \bfrak n + \cfrak n^2) +
(\bfrak + \cfrak) + \big(2\cfrak+\nu(\rbb_+)\big)n +
\int_{\rbb_+} x Q_n(x) \D\nu(x)
   \\
& \Le \hat\lambda_n + (\bfrak + \cfrak) +
\big(2\cfrak+\nu(\rbb_+)\big)n + (\vartheta -1)
\int_{\rbb_+} Q_n \D\nu
   \\
& \hspace{-1.5ex}\overset{\textrm{(iii-b)}}\Le
\vartheta \hat\lambda_n, \quad n \in \zbb_+,
   \end{align*}
we see that $\lambda_n^2 \Le \vartheta$ for every
$n\in \zbb_+$. This combined with \eqref{tykwe} shows
that $1 + \tau \Le \lambda_n^2 \Le 1 + M$ for all
$n\in \zbb_+$ with $M:=\vartheta -1$. It follows from
\mbox{(iii-b)} that $\inf{\supp{\nu}} >1$, so $M
> 0$, and $(1-\tau) (1+M)<1$. By Lemma~\ref{dfawyr}(i),
$\wlam$ is of type III and is similar to a subnormal
unilateral weighted shift.
   \end{proof}
We conclude this subsection by showing that there are
many CPD unilateral weighted shifts satisfying any of
the conditions (i)-(iii) of Theorem~\ref{ineqsuf}.
   \begin{exa}  \label{3uwre}
We discuss three cases $1^{\circ}$-$3^{\circ}$
corresponding to the conditions (i)-(iii) of
Theorem~\ref{ineqsuf}.

$\bold{1^{\circ}}$ Let us take any $b,c \in \rbb_+$
such that $b + c \Ge 1$. We will give, by considering
three cases, necessary and sufficient conditions for
the existence of $\alpha\in [1,\infty)$ satisfying the
inequality $2c + \alpha (1 - c) \Ge b$. The case $c <
1$ is obvious. If $c=1$, then such $\alpha$ exists if
and only if $b \Le 2c$. Finally, if $c>1$, then such
$\alpha$ exists if and only if $\frac{2c-b}{c-1} \Ge
1$, or equivalently $b \Le c+1$. Now, take any
compactly supported Borel measure $\nu$ on $\rbb_+$
such that $\nu(\{1\})=0$, $\nu(\rbb_+)=\alpha$ and
$\inf{\supp{\nu}} \Ge 2(1+c)$. It follows from
Theorem~\ref{wkwcpdws} that there exists a CPD
unilateral weighted shift $\wlam$ with the
representing triplet $(\bfrak, \cfrak, \nu)$ such that
$\bfrak=b$ and $\cfrak=c$ which satisfies the
condition (i) of Theorem~\ref{ineqsuf}.

$\bold{2^{\circ}}$ Let us take $b,c\in\rbb_+$. Note
that there is a unique solution $t_0\in (0,\infty)$ of
the quadratic equation $1 - 2t -\frac{3}{2}t^2=0$.
Take any $t \in (0,t_0)$. Then $1 - 2t
-\frac{3}{2}t^2> 0$. This implies that there exists
$\alpha_0 \in (0,\infty)$ such that for every $\alpha
\in [\alpha_0, \infty)$, $b + c + \alpha t (1 + t) \Ge
1$, $2 c + \alpha(1 - 2t -\frac{3}{2}t^2) \Ge b $ and
$\alpha t(2+t) \Ge 2 c$. With the parameters
$b,c,t,\alpha$ fixed, we can proceed as in the
previous paragraph to get $\wlam$ satisfying condition
(ii) of Theorem~\ref{ineqsuf} (but now we require that
$\inf{\supp{\nu}} \Ge 2+t$).

$\bold{3^{\circ}}$ First, we consider the case of
$\cfrak=0$. Let us take $\tau \in (\frac{2}{3}, 1)$.
Then $(\frac{4}{3}, 2\tau) \neq \emptyset$. Let $t \in
(\frac{4}{3}, 2\tau)$. Note that $J_1:=\big(1+
\frac{\tau}{1-\frac{t}{2}}, \frac{1}{1-\tau}\big) \neq
\emptyset$. Take $\theta \in J_1$ and note that
$J_2:=\big(\frac{2\tau^2}{2-t}, (\theta-1)\tau\big)
\neq \emptyset$. Take any $\alpha \in J_2$. We will
show that
   \begin{align} \label{trecz}
\theta > 1 + \tau + t.
   \end{align}
Indeed, since $t > \frac{4}{3}$ and $\tau >
\frac{t}{2}$, we see that $\tau + t > 2$. It is easy
to see that the latter inequality is equivalent to $1+
\frac{\tau}{1-\frac{t}{2}} > 1 + \tau + t$. This
implies \eqref{trecz}. Set $b=\tau$ and $c=0$. This
all together implies that $\theta -1
> b = \tau$  (use \eqref{trecz}), $\alpha(1 -
\frac{t}{2}) > \tau b$, $\alpha t > 0$, $\alpha <
(\theta -1) b$ and $(1-\tau)\theta < 1$. Having chosen
the parameters $b, c, t, \tau, \alpha, \theta$ and
using \eqref{trecz}, we can repeat the procedure used
in the previous two paragraphs, leading to the measure
$\nu$ with $\nu(\rbb_+)=\alpha$, $\inf{\supp{\nu}} \Ge
1+\tau+t$ and $\sup{\supp{\nu}}=\theta$, and
consequently to a CPD unilateral weighted shift
$\wlam$ satisfying condition (iii) of
Theorem~\ref{ineqsuf} with $\cfrak = 0$. Now, using
the limit argument, we conclude that if $c\in
(0,\infty)$ is small enough, then the above implies
that $\theta -1 > b+c > \tau$, $2c + \alpha(1 -
\frac{t}{2}) > \tau b$, $\alpha t > 2 \tau c$, $2c +
\alpha < (\theta -1) b$ and $(1-\tau)\theta < 1$. As a
consequence, this leads to a CPD unilateral weighted
shift $\wlam$ satisfying condition (iii) of
Theorem~\ref{ineqsuf} with $\cfrak > 0$.
   \hfill $\diamondsuit$
   \end{exa}

\appendix
\numberwithin{equation}{section}
\section{Quasi-similarity versus similarity}
   The following lemma characterizes when,
given two unilateral weighted shifts, one of
them is a quasi-affine transform of the
other. For the sake of completeness, we will
sketch the proof (see
\cite[Problem~90]{Hal82} for the case of
similarity).
   \begin{alem} \label{wpwnc}
Let $\wlam\in \ogr{\ell^2}$ and
$W_{\omegab}\in \ogr{\ell^2}$ be unilateral
weighted shifts. Then the following
conditions are equivalent{\em :}
   \begin{enumerate}
   \item[(i)] there exits $X \in \ogr{\ell^2}$
such that $\overline{\ob{X}} = \ell^2$ and
$X\wlam = W_{\omegab}X$,
   \item[(ii)]
$\wlam$ is a quasi-affine transform of
$W_{\omegab}$,
   \item[(iii)] $\sup_{n\in \zbb_+}
\frac{\hat\omega_n} {\hat\lambda_n} <
\infty$.
   \end{enumerate}
   \end{alem}
   \begin{proof}
(i)$\Rightarrow$(iii) Setting
$a_{i,j}=\is{Xe_j}{e_i}$ for $i,j \in
\zbb_+$, we get
   \begin{align} \label{vczq}
\lambda_j a_{i+1,j+1} = \is{X\wlam
e_j}{e_{i+1}} = \is{W_{\omegab}X
e_j}{e_{i+1}} = \omega_i a_{i,j}, \quad
i,j\Ge 0.
   \end{align}
Similarly, we get
   \begin{align} \label{prdv}
\lambda_j a_{0,j+1} = 0, \quad j \Ge 0.
   \end{align}
Using an induction argument, we infer from
\eqref{vczq} that
   \begin{align}  \label{freq}
a_{n,n} = \sqrt{\frac{\hat \omega_n}{\hat
\lambda_n}} a_{0,0}, \quad n \Ge 0.
   \end{align}
Note that $a_{0,0} \neq 0$ (because
otherwise, by \eqref{prdv},
$\is{Xe_j}{e_0}=0$ for all $j\in \zbb_+$, so
$e_0 \perp \overline{\ob{X}}=\ell^2$, a
contradiction). Then \eqref{freq} yields
$\sup_{n\in \zbb_+} \frac{\hat
\omega_n}{\hat \lambda_n} \Le
\frac{\|X\|^2}{|a_{0,0}|^2}$.

(iii)$\Rightarrow$(ii) There exists a unique
$X\in \ogr{\ell^2}$ such that
$Xe_n=\sqrt{\frac{\hat \omega_n}{\hat
\lambda_n}} e_n$ for $n\in \zbb_+$. Clearly,
$\jd{X}=\{0\}$ and
$\overline{\ob{X}}=\ell^2$. It is easy to
see that $X\wlam = W_{\omegab}X$.

(ii)$\Rightarrow$(i) Trivial.
   \end{proof}
Now we show that for unilateral weighted
shifts the quasi-similarity and similarity
relations coincide.
   \begin{athm}
Let $\wlam\in \ogr{\ell^2}$ and
$W_{\omegab}\in \ogr{\ell^2}$. Then the
following conditions are equivalent{\em :}
   \begin{enumerate}
   \item[(i)] there exit operators $X,Y \in \ogr{\ell^2}$
with dense range such that $X\wlam =
W_{\omegab}X$ and $Y W_{\omegab}=\wlam Y$,
   \item[(ii)]
$\wlam$ and $W_{\omegab}$ are quasi-similar,
   \item[(iii)] $\wlam$ and $W_{\omegab}$ are
similar,
   \item[(iv)] $0 < \inf_{n\in
\zbb_+} \frac{\hat\omega_n} {\hat\lambda_n}$
and $\sup_{n\in \zbb_+}\frac{\hat\omega_n}
{\hat\lambda_n} < \infty$.
   \end{enumerate}
   \end{athm}
   \begin{proof}
Implication (i)$\Rightarrow$(iv) is a direct
consequence of Lemma~\ref{wpwnc}.
Implication (iv)$\Rightarrow$(iii) is well
known (see \cite[Problem~90]{Hal82}); see
also the proof of (iii)$\Rightarrow$(ii).
Implications (iii)$\Rightarrow$(ii) and
(ii)$\Rightarrow$(i) are trivial.
   \end{proof}

   \bibliographystyle{amsalpha}
   
   \end{document}